[1994/12/01]
\documentclass{amsart}

\usepackage{amsmath}
\usepackage{amsthm}
\usepackage{amssymb}
\usepackage{enumitem}
\usepackage{mathrsfs}
\usepackage{mathrsfs}
\usepackage{hyperref}

\newtheorem{theorem}{Theorem}[section]
\newtheorem{theorem*}{Theorem}
\newtheorem{corollary}[theorem]{Corollary}
\newtheorem{lemma}[theorem]{Lemma}
\newtheorem{proposition}[theorem]{Proposition}
\newtheorem{question}{Question}

\theoremstyle{definition}
\newtheorem{definition}[theorem]{Definition}
\newtheorem{remark}[theorem]{Remark}
\newtheorem{fact}[theorem]{Fact}
\newtheorem{example}[theorem]{Example}

\theoremstyle{remark}
\newtheorem{observation}[theorem]{Observation}
\newtheorem{step}{Step}
\newtheorem*{claim*}{Claim}
\newtheorem{claim}{Claim}

\newcommand{\abs}[1]{\left|#1\right|}

\newcommand{\norm}[1]{\left\lVert #1 \right\rVert}

\newcommand{\bbra}[1]{ { \left\{ #1 \right\} } } 
\newcommand{\bra}[1]{ \left( #1 \right) }

\newcommand{\nobar}[1]{#1}

\newcommand{\cP}{\mathcal{P}}

\newcommand{\cC}{\mathscr{C}}
\newcommand{\cD}{\mathscr{D}}

\newcommand{\cA}{\mathcal{A}}

\newcommand{\cF}{\mathscr{F}}
\newcommand{\cFe}{\cF_{\emptyset}}

\newcommand{\fpa}[1]{\left\lVert #1 \right\rVert_{\mathbb{R}/\mathbb{Z}}}

\renewcommand{\a}{\alpha}
\renewcommand{\b}{\beta}
\renewcommand{\c}{\gamma}
\newcommand{\e}{\varepsilon}

\newcommand{\Borel}{\mathscr{B}}

\newcommand{\NN}{\mathbb{N}}
\newcommand{\QQ}{\mathbb{Q}}

\newcommand{\ZZ}{\mathbb{Z}}
\newcommand{\RR}{\mathbb{R}}
\newcommand{\TT}{\mathbb{T}}

\newcommand{\set}[2]{\left\{ #1 \ \middle| \ #2 \right\} }

\newcommand{\G}{\Gamma}

\newcommand{\bchecked}[1] {}

\newcommand{\parbreak}[1]{
\begin{center}
***
\end{center}
}

\newcommand{\nilbohr}[1]{$\mathrm{Nil}_{#1}\!\text{--}\mathrm{Bohr}$}
\newcommand{\nilbohrz}[1]{$\mathrm{Nil}_{#1}\!\text{--}\mathrm{Bohr}_0$}

\newcommand{\bohrz}{$\mathrm{Bohr}_0$}
\newcommand{\bohr}{$\mathrm{Bohr}$}

\newcommand{\ip}{$\mathrm{IP}$}
\newcommand{\ipd}{$\mathrm{IP}^*$}
\newcommand{\sg}[1]{$\mathrm{SG}_{#1}$}
\newcommand{\sgd}[1]{$\mathrm{SG}_{#1}^*$}

\newcommand{\SG}{\operatorname{SG}}
\newcommand{\FS}{\operatorname{FS}}
\newcommand{\FU}{\operatorname{FU}}
\newcommand{\IP}{\operatorname{IP}}
\newcommand{\HK}{\operatorname{HK}}
\newcommand{\poly}{\operatorname{poly}}
\newcommand{\codim}{\operatorname{codim}}
\newcommand{\diam}{\operatorname{diam}}
\newcommand{\cl}{\operatorname{cl}}

\newcommand{\cS}{\mathcal{S}}

\renewcommand{\subset}{\subseteq}
\renewcommand{\supset}{\supseteq}

\newcommand{\plim}[2]{#1\text{--}\lim_{#2}}
\newcommand{\bb}{\mathbf}	

\newcounter{tmp}

\begin{document}

\author[J. Konieczny]{Jakub Konieczny
}
\address{Mathematical Institute, 
University of Oxford\\
Andrew Wiles Building,
Radcliffe Observatory Quarter\\
Woodstock Road,
Oxford,
OX2 6GG}
\email{jakub.konieczny@gmail.com}

\begin{abstract}
In this paper we study the relation between two notions of largeness that apply to a set of positive integers, namely \nilbohrz {d}\ and \sgd {k}, as introduced by Host and Kra \cite{HostKra2009}. We prove that any \nilbohrz {d}\ set is necessarily \sgd {k} where ${k}$ is effectively bounded in terms of $d$. This partially resolves a conjecture of Host and Kra.  
\end{abstract}

\title[Nil--Bohr sets and combinatorics]{Combinatorial properties of Nil--Bohr sets}

\maketitle 

\section{Introduction}\label{sec:INT} 

Among the basic problems in additive combinatorics is the study of various notions of \emph{largeness} which may apply to a set of integers. In this paper we are specifically interested in one such notion, namely that of being a \nilbohr {d} set, or a set of recurrence times for  a $d$-step nilrotation (see Section \ref{sec:INT:NilBohr}).

The study of these sets was pioneered by Host and Kra \cite{HostKra2009}, with later developments due to Huang, Shao and Ye \cite{HuangShaoYe2014}, Tu \cite{Siming}, and Bergelson and Leibman \cite{BergelsonLeibman2016}. In \cite{HostKra2009}, it was realised that \nilbohr {d} sets bear a striking relation to a purely combinatorial class of \sgd {k} sets (see Section \ref{sec:INT:SG}). Namely, it was shown that a \sgd {d} set is (strongly) piecewise-\nilbohr {d}. Here, we prove the reverse implication, although in a weaker form.

Even though a proper motivation for our results requires more context, we are able to express some of them in relatively basic terms. Our first result is the following.

\begingroup
\setcounter{tmp}{\value{theorem}}
\setcounter{theorem}{0} 
\renewcommand\thetheorem{\Alph{theorem}}

\begin{theorem}\label{thm:A}
	Fix an polynomial $p \in \RR[x]$ of degree $d$ with $p(0) = 0$, and a sequence $\bra{n_i}_{i=1}^\infty$ of positive integers. Then, for any $\e > 0$, there exists a finite, non-empty set $\a \subset \NN$, with gaps bounded by $d$, such that $\fpa{p\bra{\sum_{i \in \a} n_i}} \leq \e$.  
\end{theorem}

 Above, $\fpa{t} = \min_{n \in \ZZ} \abs{t-x}$ denotes the distance from the closest integer. From now on, let us denote by $\cF$ the family of all finite non-empty subsets of $\NN = \{1,2,\dots\}$. It is customary to denote elements of $\cF$ by lowercase Greek letters $\a, \b, \dots$. Bootstrapping (a slight modification of) the above result, we obtain a recurrence statement for nilrotations. For an set $\a = \{i_1,i_2,\dots,i_r\}$, where $i_1 < i_2 < \dots < i_r$, the \emph{gaps} of $\a$ are the integers $i_2 - i_1, i_3-i_2,\dots,i_{r}-i_{r-1}$. 

\begin{theorem}\label{thm:B}
	Let $G$ be a $d$-step nilpotent Lie group and let $\Gamma < G$ be a cocompact, discrete subgroup. Fix $g \in G$, an open neighbourhood $e \Gamma \in U \subset G/\Gamma$, as well as a sequence $\bra{n_i}_{i=1}^\infty$ of positive integers.
	
	Then, there exists $\a \in \cF$ with gaps bounded by $d' = 4d$, such that $g^{\sum_{i \in \a}\! n_i}\Gamma \in U $.
\end{theorem}
\endgroup
\setcounter{theorem}{\thetmp} 

The bound $d' = 4d$ results from an inductive argument. By a marginally more careful computation, this could be improved to $d' = 3d + \log_2{d} + O(1)$, but we sacrifice this inconsequential improvement for the sake of readability. The optimal value is believed to be $d' = d$, but our argument does not yield this conclusion.

Finally, we point out that the recent paper of Bergelson and Leibman \cite{BergelsonLeibman2016} proves highly relevant results. It follows as a special case of Theorem 0.3 in \cite{BergelsonLeibman2016} that, in the situation of Theorem \ref{thm:A}, there exists a set $\a \subset [r]$ such that $\fpa{p\bra{\sum_{i \in \a} n_i}} \leq \e$, for some $r = r(d,\e)$. Hence, trivially, we may ensure that $\a$ has gaps bounded in terms of $d$ and $\e$; however, our result gives a good bound on the gaps, which also is uniform in $\e$. Similarly, a version of Theorem \ref{thm:B} can be read off Theorem 0.5 in \cite{BergelsonLeibman2016}.

\subsection{Bohr sets}\label{sec:INT:Bohr}
 
The notion of a \bohr\ (or \bohrz) set is classical  and well-studied. A set $A \subset \NN_0$ is said to be a \bohr\ set if it contains the preimage of an open, non-empty set $U$ through the natural embedding of $\NN_0$ in the Bohr compactification of $\ZZ$, usually denoted $b\ZZ$. Accordingly, $A$ is a \bohrz\ set if additionally $0 \in U$.

While very satisfying from the categorical point of view, the above definition gives limited idea of what a \bohr\ set looks like. A more concrete description is possible. Namely, a set is \bohr\ if it contains a non-empty set of the form $\{ n \in \NN_0 \ : \ n \a \in U\}$ where $\a \in \TT^m = \RR^m/\ZZ^m$ and $U \subset \TT^m$ is open; $A$ is \bohrz\ if additionally $0 \in U$. Hence, \bohrz\ sets can be viewed dynamically as a return-times sets for the point $0 \in \TT^m$, where the dynamics are given by $x \mapsto x + \a$.

(Note that we construe being a \bohr\ set as a notion of largeness, hence above we only insist on containment, rather than equality. In context when precise structure is important, different definitions are used, see e.g.\ \cite[Section 4.4]{TaoVu}.)

\subsection{Nil--Bohr sets}\label{sec:INT:NilBohr}

With the advent of higher-order Fourier analysis, a natural analogue of the class of \bohr\ sets has come into view. The role of the circle rotations in classical Fourier analysis is now played by \emph{nilrotations}, which we presently define.

Suppose that $G$ is a $d$-step nilpotent Lie group, and let $\Gamma < G$ be a cocompact and discrete subgroup. Here, by \emph{cocompact} we simply mean that the quotient space $G/\Gamma$ should be compact. The space $X = G/\Gamma$ is a $d$-step \emph{nilmanifold} and carries a natural action of $G$, given by $g.x\Gamma = (gx)\Gamma$. There exists a unique Haar measure $\mu$ on the Borel $\sigma$-algebra $\Borel(X)$, which is preserved by all translations $T_g\colon x\Gamma \mapsto g.x\Gamma$. Hence, for any $g \in G$, translation by $g$ is a measure-preserving transformation of $(X,\mu)$. We call any such system $(X,\Borel(X),\mu,T_g)$ a $d$-step \emph{nilrotation}. 

We now define a set $A \subset \NN_0$ to be a \nilbohr {d} set, in analogy to the abelian case, if it contains a non-empty set of the form
$$
	\{ n \in \NN_0 \ : \ g^n \Gamma \in U \}
$$
where $G/\Gamma$ is a $d$-step nilmanifold, and $U \subset G/\Gamma$ is open. If additionally $e \Gamma \in U$ then $A$ is a \nilbohrz {d} set.

A useful example to keep in mind are sets of the form
$$
	\{ n \in \NN_0 \ : \ \fpa{ p(n) } \leq \e \},
$$
where $p \in \RR[x]$ is a polynomial with at least one irrational non-constant coefficient. In general, such sets are \nilbohr {d}. If additionally $p(0) = 0$ then they are \nilbohrz {d}. (This can be seen by a classical construction, which is discussed for instance in \cite[Section 0.16]{BergelsonLeibman2007}.)

\subsection{IP sets}\label{sec:INT:IP}

Another classical notion of largeness which is relevant to us is \ip. For a sequence $(n_i)_{i=1}^\infty$, we define the set of finite sums of $A$,
\begin{equation}
	\FS(n_i) = \bbra{ \sum_{i \in \a} n_i \ : \ \a \in \cF }.
	\label{INT:eq:def-of-FS}
\end{equation}
For brevity of notation, it is convenient to define in this context $n_\a := \sum_{i \in \a} n_i$ for $\a \in \cF$; this is consistent with the natural inclusion $\NN \ni i \mapsto \{i\} \in \cF$.

A set $A \subset \NN_0$ is said to be an \ip\ set if there exists a sequence $(n_i)_{i \in \NN}$ such that $\FS(n_i) \subset A$. Once again, we remark that since we view \ip\ as a notion of largeness, we only require $A$ to contain a set of finite sums (as opposed to being equal to such a set). This is consistent with usage e.g.\ in \cite{BergelsonFurstenbergWeiss2006}, but different from the original definition in \cite{FurstenbergWeiss1978}.

\subsection{SG sets} \label{sec:INT:SG}

In analogy to the \ip\ sets $\FS(n_i)$, we define for $k \geq 0$ the sets $\SG_{k}(n_i)$, where the index sets are additionally required to have bounded gaps. Let $\cS_{k} \subset \cF$ denote the set of finite sets of integers $\a \in \cF$ whose gaps are bounded by $k$; we might call such sets $k$-syndetic. In other words, we require that for any $i \in \a$, either $i = \max \a$, or there exists $j \in \a$ with $i < j \leq i+{k}$. We allow the degenerate case $k = 0$, where $\cS_0 = \NN$ (up to identification $i \mapsto \{i\}$). For a sequence of integers $(n_i)_{i=1}^\infty$ we put
\begin{equation}
	\SG_{k}(n_i) = \bbra{ \sum_{i \in \a} n_i \ : \ \a \in \cS_{k} }.
	\label{INT:eq:def-of-SG}
\end{equation}
In analogy with \ip\ sets, we define a set $A \subset \NN_0$ to be a \sg {k} set if it contains a set of the form $\SG_{k}(n_i)$ for some sequence $n_i$. To the best of our knowledge, this definition first appears in \cite[Def. 2.9]{HostKra2009}. We have an obvious chain of inclusions $\SG_1(n_i) \subset \SG_2(n_i) \subset \dots \subset \FS(n_i)$, whence $\SG_1\! \supset \SG_2\! \supset \dots \supset \IP$. 

The simplest non-degenerate example, though possibly a misleading one, is when ${k} = 1$. Then $\SG_{k}(n_i)$ consists precisely of the consecutive sums $\sum_{i=u}^v n_i$, and it has been noted that $\SG_{1}(n_i) = \Delta(S) := (S - S) \cap \NN$ where $S = \{ \sum_{i=1}^v n_i \ : \ v \in \NN \}$. It is not difficult to see that conversely, any set of the form $\Delta(S)$ as above can be expressed as $\SG_{1}(n_i)$ for some sequence $(n_i)$. Thus, \sg {1} sets coincide with the well-studied class of $\Delta$ sets (see also \cite{BergelsonFurstenbergWeiss2006}).

\subsection{Dual classes}\label{sec:INT:Dual}
For a class $\cC$ of subsets of $\NN_0$, we define the \emph{dual} class $\cC^*$ by declaring that $B \in \cC^*$ if and only if for any $A \in \cC$ the sets $A$ and $B$ intersect non-trivially: $A \cap B \neq \emptyset$ (see e.g.\ \cite[Section 9.1]{Furstenberg1981}).

Specifically, we define the class \sgd {k}, consisting of the sets $B \subset \NN_0$ such that for any choice of integers $(n_i)_{i=1}^\infty$, there exists some $\a \in \cF$ with $n_\a \in B$. We note the reversed chain of inclusions: $\SG_1^* \subset \SG_2^* \subset \dots \subset \IP^*$.

It is clear by definition that for any class $\cC$, the dual class $\cC^*$ is closed under taking supersets. The operation of taking the dual reverses the inclusion: if $\cC \subset \cD$ then $\cC^* \supset \cD^*$. If $\cC$ is partition regular, then $\cC^*$ is easily seen to be closed under finite intersections, but $\cC^*$ will not generally be partition regular \cite[Lemma 9.5]{Furstenberg1981}. If additionally $\emptyset \not \in \cC$ then $\cC \subset \cC^*$ \cite[Lemma 9.4]{Furstenberg1981}. We cite the latter two facts merely to provide context, they are not used at any point.

\subsection{Nil--Bohr sets vs.\ SG sets}\label{sec:INT:SG-vs-NilBohr}
As noted earlier, there is a somewhat unexpected connection between the a priori unrelated notions of \sgd{k} sets and  \nilbohr{d} sets. 

Following the usual convention, we say that a set $A$ is piecewise \nilbohrz {d} if there exists a thick set $T$ (i.e. such that $T$ contains arbitrarily long intervals) and a \nilbohrz {d} set $B$ such that $A = B \cap T$. It was proved in \cite[Theorem 2.10]{HostKra2009} that any \sgd{d} set is piecewise \nilbohrz {d}. (In fact, the result proved there is stronger, with a more rigid notion of ``strongly piecewise''.) The following question arises naturally:

\begin{question}\label{question:basic}
	Is any \nilbohrz {d} set a \sgd {d} set? 
\end{question}

The main purpose of this paper is to answer a weaker variant of this question. The reader will have no problem checking that the following is merely a succinct restatement of Theorem \ref{thm:B}.

\begin{theorem}\label{thm:main}\label{thm:B-short}
	Any \nilbohrz {d} set is \sgd {k}, provided that $k \geq 4d$.
\end{theorem}

\begin{remark}
We note that a weaker variant of theorem, with \sgd {k} replaced with \ipd, is true for much simpler (or at least better studied) reasons. It can be checked that any nilrotation is distal (\cite[Chpt IV, Sec. 7]{AuslanderGreenHahn1963}; see also \cite{Keynes1966},\cite{Keynes1967}). More explicitly, for any nilrotation $(T_g, G/\Gamma)$ with metric $d_{G/\Gamma}$, for any $\e > 0$ there exists $\delta > 0$ such that if $x,y \in G/\Gamma$ are two points with $d_{G/\Gamma}(x,y) > \e$ then for all $n$ also $d_{G/\Gamma}(T^n_g x, T^n_g y) > \delta$. On the other hand, for any topological dynamical system $(T,X)$ distality is equivalent to the condition that for any $x \in X$ and any open neighbourhood $U \ni x$, the set $\set{ n \in \NN }{T^n x \in U }$ is \ipd\ \cite[Theorem 9.11]{Furstenberg1981}. Since any \nilbohrz {d} is (a superset of) a set of precisely this form, the claim follows. (Essentially the same argument can be found in \cite[Sec. 0.17]{BergelsonLeibman2007}, and in \cite{BergelsonLeibman2016}.)

\end{remark}

\subsection*{Notation}

By $\NN$ we denote the set $\{1,2,3,\dots\}$; in particular $0 \not \in \NN$. We put $\NN_0 = \NN \cup \{0\}$.

By $\cF$ we denote the partial semigroup of the finite, non-empty subsets of integers, where the operation is the disjoint union. We also put $\cFe = \cF\cup \{\emptyset\}$. Hence, whenever the symbol $\a \cup \b$ is used for $\a, \b \in \cF$ it is implicitly assumed that $\a$ and $\b$ need to be disjoint.

If $G$ is a group equipped with a metric, then $\norm{g} = \norm{g}_{G}$ denotes the distance from $g$ to $e_G$. In particular, for $x \in \RR$, $\norm{x}_{\RR/\ZZ}$ denotes the distance of $x$ from the closest integer.

Standard asymptotic notation, such as $O(\cdot)$ and $\Omega(\cdot)$, is occasionally used: $X = O(Y)$ if $\abs{X} \leq c Y$ for an absolute constant $c$, and $X = \Omega(Y)$ if $0 < Y = O(X)$.

\subsection*{Acknowledgements}
The author thanks Freddie Manners for his generous help with the finer points of the theory of polynomial mappings to nilmanifolds, and Ben Green for much useful advice during the work on this project. Thanks go also to Sean Eberhard, Rudi Mrazovi\'{c}, Przemek Mazur and Aled Walker for many informal discussions. Finally, the author is grateful to the anonymous referee for his insightful comments, which helped improve the presentation of this paper.

This research was partially supported by the National Science Centre (NCN) under grant no. 2012/07/E/ST1/00185. The author also acknowledges the generous support from the Clarendon Fund and SJC Kendrew Fund for his doctoral studies.

\section{Polynomial maps} \label{sec:DEF} 

In order to prove Theorem \ref{thm:B}, we will need a good understanding of ``linear'' maps of the form $n \mapsto g^n \Gamma$ from $\NN$ to a nilmanifold $G/\Gamma$, or more accurately $\a \mapsto g^{n_\a} \Gamma$ where $\a \in \cF$ and $\a \mapsto n_\a = \sum_{i \in \a} n_i$. For reasons which will become apparent shortly, it is more natural to work with a larger class of \emph{polynomial} sequences $\NN \to G/\Gamma$ or $\cF \to G/\Gamma$ respectively, which we will shortly define.

Systematic study of polynomial sequences $\ZZ \to G$ was initiated by Leibman \cite{Leibman1998} (the results generalise easily to sequences $H \to G$ for abelian $H$), although some early results were obtained by Lazard \cite{Lazard-1954} and others. Polynomial sequences $H \to G$ for $H$ nilpotent are studied in \cite{Leibman2002}. The notion of a polynomial sequence whose domain is a partial semigroup, such as $\cF$, does not explicitly appear until later, but can easily be gleaned from \cite{Leibman1998}. From a slightly different perspective, polynomial sequences $\cF \to G$ appear in \cite{BergelsonLeibman-2003}.

Finally, the idea of measuring the ``degree'' of a polynomial sequence by means of a filtration appears in \cite{GreenTaoZiegler-2012} in context of sequences $\ZZ \to G$ (or $\ZZ \to G/\Gamma$), although most results straightforwardly generalise to different domains such as $\cF$. 

To the best of our knowledge, the precise definition of a polynomial sequence $\cF \to G$ we shall use first appears in print in the work of Zorin-Kranich \cite{Zorin-Kranich2012, Zorin-Kranich2013}. For an accessible introduction to polynomial sequences, we refer to \cite{Green-book} or \cite{Tao-book}.

\subsection{Filtered groups and polynomial sequences}
\label{sec:DEF:Filtrations}
\label{sec:DEF:Polynomials}

A Lie \emph{prefiltration} on a Lie group $G$ is a descending sequence $G_\bullet$ of Lie subgroups of $G$ such that $G = G_0 \supseteq G_1 \supseteq G_2 \supseteq \dots $ and for any $i,j$ we have the nesting condition $[G_i,G_j] \subseteq G_{i+j}$. We always assume that the prefiltration terminates at some point, in the sense that for some $d$ we have $G_{d+1} = \{e_G\}$. The least such $d$ is the \emph{length} of the filtration. A Lie \emph{filtration} is a prefiltration $G_\bullet$ such that additionally $G_0 = G_1 = G$. 
Since we have no need to consider filtrations which are not Lie, we will usually omit this adjective. We may also keep the group $G$ implicit, and assume that $G = G_0$.

The most important example of a filtration to bear in mind is the lower central series given by $G_0 = G_1 = G$ and $G_{i+1} := [G,G_i]$ (where $[G,H]$ denotes the group {generated by} elements of the form $[g,h]$ with $g \in G,\ h \in H$). Each of these groups $G_i$ is normal. If $G$ is $d$-step nilpotent, then the length of this filtration is $d$. Conversely, if $G$ has a filtration of length $d$, then $G$ is at most $d$-step nilpotent. (Of course, in general the length of a filtration may be greater than the nilpotency class of the group.) 

If $G_\bullet$ is a prefiltration, we may construct a prefiltration $G'_{i} := G_{i+1}$ ($i \geq 0$) on $G_1$. (Note that the nesting property is clear.) We denote this new prefiltration by $G_{\bullet+1}$. Images and preimages of (pre-)filtrations (by morphisms of Lie groups) are again (pre-)filtrations of (at most) the same length.

As alluded to before, we wish to study maps from $\cF$ (or $\cFe$) into a nilmanifold $G/\Gamma$, taking a rather specific form $\a \mapsto g^{n_\a}\Gamma$. It will be convenient to introduce a more general notion of a polynomial sequence from a general partial semigroup into a filtered group.

\renewcommand{\cA}{\mathcal{A}}
\renewcommand{\a}{\alpha}
\renewcommand{\b}{\beta}
\renewcommand{\c}{\gamma}

Recall that a partial semigroup $(\cA,*)$ is a set $\cA$ equipped with a binary operation $*$ defined of a subset of $\cA \times \cA$, such that $\a * (\b * \c) = (\a * \b) * \c$, whenever both sides are defined. By the usual abuse of notation, we usually refer to the set $\cA$ alone as a semigroup, keeping the operation $*$ implicit. We are mostly interested in the cases $(\cA,*) = (\cFe,\cup)$ and $=(\NN_0,+)$. For $\cA = \ZZ$ the following definition appears in Leibman \cite[Sec. 1.4]{Leibman1998}, and for a fully general version see \cite[Def.\ 1.21]{Zorin-Kranich2013}.

\begin{definition}[Polynomial sequence]
\label{def:poly(A->G)}
	Let $\cA$ be an partial semigroup, and let $G_{\bullet}$ be a prefiltration on a nilpotent Lie group $G$. Then a map $g \colon \cA \to G$ is declared to be polynomial (with respect to $G_{\bullet}$) if either $G = \{e_G\}$ is trivial and $g(\a) = e_{G}$ is the constant sequence, or if $G$ is non-trivial and for every $\beta \in \cA$ there exists a $G_{\bullet + 1}$ polynomial $D_{\beta} g$ such that 
	\begin{equation}	
		D_{\beta}g(\a) = g(\a)^{-1} g(\a * \b) 		\label{DEF:eq:def-of-D-x}
	\end{equation}
	whenever $\a * \b$ is defined. The set of all such polynomial maps is denoted by $\poly(\cA \to G_{\bullet})$. The length of the prefiltration my occasionally be referred to as the degree of the polynomial.
\end{definition}

When $\cA = \cFe$, and $g \in \poly(\cA \to G_\bullet)$ is such that $g(\emptyset) = e_G$, then the symmetric derivative
	\begin{equation}	
		\Delta_{\beta}g(\a) = D_\beta f(\a) g(\b)^{-1} = g(\a)^{-1} g(\a \cup \b) g(\b)^{-1} \qquad 	 (\text{for } \b \cap \a = \emptyset) \label{DEF:eq:def-of-Delta-x}
	\end{equation}
	is a polynomial with respect to $G_{\bullet+1}$, which has the added advantage that $\Delta_\beta g(\emptyset) = e_G$. This property can also be used for as definition of polynomials, when we restrict to maps with $g(\emptyset) = e_G$. The analogous remark applies to maps $\NN_0 \to G$. (See \cite[Sec. 1.3]{Zorin-Kranich2013} for details.)

\begin{example}\label{DEF:ex:2}
	If $G$ is $d$-step nilpotent, then the sequence $\NN \to G$ given by $g(n) = a^n$, ($a \in G$), is polynomial with respect to the lower central series: indeed, $D_{m} g(n) = a^{m}$ and $D_{m_2} D_{m_1} g(n) = e$. We will shortly see that the sequence $g'(n) = a^n b^n$ ($a,b \in G$) is also polynomial with respect to the same filtration. 
\end{example}

\begin{example}\label{DEF:ex:3}
	Let $G = \RR$, equipped with length $d$ filtration $G_0 = G_1 = \dots = G_d = \RR$, $G_{d+1} = \dots = \{0\}$. Then polynomial sequences $\NN \to \RR$ are precisely the polynomials in the conventional sense, that is sequences of the form $p(n) = \sum_{i=1}^d n^i a_i$ where $a_i \in \RR$.
	
	By the same token, for any sequence of integers $(n_i)$ and polynomial $p \colon \NN \to \RR$, the map $\cF \ni \a \mapsto p\bra{n_\a} \in \RR$ is polynomial with respect to the aforementioned filtration.  
\end{example}

\begin{example}\label{DEF:ex:4}
	Let $G_{\bullet}$ be a prefiltration on a nilpotent Lie group $G$, and let $g \in \poly(\NN \to G_\bullet)$. Then, for any sequence $(n_i)_i$, the map $\cF \ni \a \mapsto g(n_\a) \in G$ is polynomial. This is a special instance of a general fact that composition of a polynomial with a morphism of partial semigroups is again a polynomial. 
\end{example}

A principal advantage of working with polynomial maps is that they are closed under products. In a slightly more specific context, this is the main result of \cite{Leibman1998}, in this form it appears in \cite[Thm.\ 1.23]{Zorin-Kranich2013} (see also Section \ref{sec:DEF:HostKra}).

\begin{theorem}[Lazard--Leibman]\label{thm:poly-is-group}
	Let $G_{\bullet}$ be a prefiltration and $\cA$ a partial semigroup. Then $\poly(\cA \to G_{\bullet})$ is a group under pointwise multiplication.
\end{theorem}

\subsection{Polynomials to nilmanifolds}
We are fundamentally interested not in maps $\poly(\cFe \to G_\bullet)$, but rather in their projections onto $G/\Gamma$, where $\Gamma$ is a discrete cocompact subgroup. In order for such maps to be well behaved, we need to assume that $G_\bullet$ is compatible with $\Gamma$. We will say that $G_\bullet$ is \emph{$\Gamma$-rational} if for any $i$ the discrete subgroup $\Gamma_i := G_i \cap \Gamma$ is cocompact. It was shown by Mal'cev \cite{Malcev1951} that the lower central series is $\Gamma$-rational with respect to any choice of $\Gamma$, under the additional assumption that $G$ is simply connected. 

\begin{definition}\label{def:poly(A->G/Gamma)}
	Let $\cA$ and $G_{\bullet}$ be as in Definition \ref{def:poly(A->G)}, and let $\Gamma$ be cocompact discrete subgroup of $G$ such that $G_\bullet$ is $\Gamma$-rational.
 Then a map $\bar{g} \colon \cA \to G/\Gamma$ is declared to be polynomial (with respect to $G_\bullet$) if and only if $\bar{g}$ takes the form $\bar{g} = g \circ \pi$ where $g \in \poly(\cA \to G_\bullet)$ and $\pi\colon G \to G/\Gamma$ is the standard projection $h \mapsto h \Gamma$.  The set of all such polynomial maps is denoted by $\poly(\cA \to G_{\bullet}/\Gamma)$. 
\end{definition}

It will be often be convenient to assume that the groups $G_i$ are simply connected (see Section \ref{sec:DEF:Connected}). For the sake of brevity, if $G_\bullet$ is a Lie prefiltration of length $d$ where each $G_i$ is simply connected, and $\Gamma$ is a  discrete cocompact subgroup of $G = G_0$ such that $G_\bullet$ is $\Gamma$-rational, we will say that $G_\bullet/\Gamma$ is a $d$-step nilmanifold. 

\begin{remark}
	While we restrict to the simply connected case, most of the statements will remain true also in the general case of not necessarily connected and not necessarily simply connected groups, possibly after minor modifications.
\end{remark}

We mention a crucial fact, whose proof we defer to Section \ref{sec:DEF:HostKra}.

\begin{theorem}\label{thm:poly-is-compact}
	If $G_\bullet/\Gamma$ is a $d$-step nilmanifold, then the space $\poly(\cFe \to G_{\bullet}/\Gamma)$ is compact in the topology of pointwise convergence.
\end{theorem}

We are now ready to formulate versions of Theorems \ref{thm:A} and \ref{thm:B} in the proper generality. 

\begin{theorem}[\ref{thm:A}, strong version]\label{thm:A-poly(F->T)}
	
	Let $f \colon \cFe \to \TT^m$ be a polynomial of degree $d$ with $f(\emptyset) = 0$. Then, for any $\e > 0$, there exists $\a \in \cS_{d}$, such that $\norm{ \nobar{f}(\a) } < \e$.
\end{theorem}

\begin{theorem}[\ref{thm:B}, strong version]\label{thm:B-poly(F->G/Gamma)}
	Let $G_{\bullet}/\Gamma$ be a nilmanifold of length $d$, and let $\bar{f} \in \poly(\cFe \to G_{\bullet}/\Gamma)$ with $f(\emptyset) = e \Gamma$. Then, for any open neigbourhood $e\Gamma \in U \subset G/\Gamma$, there exists $\a \in \cS_{k}$ with $k \leq 4d$, such that $\bar{f}(\a) \in U$.
\end{theorem}

\begin{remark}\label{remark:main-strong-vs-weak}
	Theorem \ref{thm:A-poly(F->T)} is ostensibly stronger than Theorem \ref{thm:A}, and there is a similar relation between Theorems \ref{thm:B-poly(F->G/Gamma)} and \ref{thm:B}. While it is beyond the scope of our investigation to ask if Theorem \ref{thm:A} formally implies \ref{thm:A-poly(F->T)}, we pause to give some examples polynomials of the special form we are interested in are already dense in the family of all polynomials.
\end{remark}

\begin{example}\label{ex:linear-approximate-general}
	Fix $\e_i >0$ and $\theta \in \RR \setminus \QQ$, and let $f \colon \cFe \to \TT = \RR/\ZZ$ be a degree $1$ polynomial with $f(\emptyset) = 0$. Then, there exist $(n_i)_{i=1}^\infty$ such that for all $\a \in \cFe$ we have $\norm{ f(\a) - n_\a \theta} < \e_\a$, where as usual $n_\a = \sum_{i \in \a} n_i$, $\e_\a = \sum_{i \in \a} \e_i$. 
\end{example}
\begin{proof}
	By Proposition \ref{prop:structure-of-poly(F->T)}, $f$ takes the form $f(\a) = \sum_{i\in \a} a_i$. Choose $n_i$ so that $\norm{a_i - n_i \theta} < \e_i$.
\end{proof}

\begin{example}\label{ex:quadratic-approximate-general}
	Fix $\e > 0$ and $r \in \NN$.
	Let $\theta \in \RR \setminus \QQ$, and let $f \colon \cFe \to \TT$ be a degree $2$ polynomial with $f(\emptyset) = 0$. Then, there exist $(n_i)_{i=1}^r$ such that for all $\a \subset [r]$ we have:
	\begin{equation}\label{eq:021}
		\fpa{f(\a) - n_\a^2 \theta} < \e,
	\end{equation}
	where as usual $n_\a = \sum_{i \in \a} n_i$.
\end{example}
\begin{proof}

Using Proposition \ref{prop:structure-of-poly(F->T)}, $f(\a) = \sum_{\gamma \subset \a} a_\gamma$ (where $a_\gamma = 0$ if $\abs{\gamma} > 2$ or $\gamma = \emptyset$). Subject to the choice of $n_i$'s, write $b_{\{i,j\}} = 2 n_i n_j \theta$, $b_{\{i\}} = n_i^2 \theta$, and $b_\gamma = 0$ otherwise, so that $n_\a^2 \theta = \sum_{\gamma \subset \a} b_\gamma$.

To obtain \eqref{eq:021}, it will suffice to ensure that $\fpa{b_\gamma - a_\gamma} < \frac{\e}{2^r}$ for all $\gamma$, for a suitable choice of $(n_i)_i$. We claim that more is true, namely that $(b_\gamma)_{\gamma}$ is equidistributed in $\TT^{\binom{r+1}{2}}$ (here, $\gamma$ runs over all non-empty subsets of $[r]$ of size $\leq 2$). If this was no so, then by the multidimensional version of Weyl's equidistibution theorem, there would exist $k_\gamma \in \ZZ$, not all $0$, such that $\sum_{\gamma \subset [r]} k_\gamma b_\gamma$ is a polynomial with integer coefficients in $(n_i)_i$ --- but this is impossible since $\theta \not \in \QQ$, and the non-trivial $b_\gamma$ are distinct monomials.
\end{proof}

	Similar statements hold for polynomials of higher degrees, the only difference being that the verification of equidistibution in the final step becomes more mundane.

\subsection{Connectivity}\label{sec:DEF:Connected}

If the nilpotent group $G$ is simply connected (which we take to mean, in particular, connected) then the exponential map $\exp\colon \mathfrak{g} \to G$ is a homeomorphism from the Lie algebra $\mathfrak{g}$ of $G$ and $G$ itself. If $G_\bullet$ is a prefiltration where each of the groups $G_i$ is connected, then there exists a Mal'cev basis for $\mathfrak{g}$, where the description of $G$ takes a particularly simple form (see \cite{Malcev1951}). Because of this, the assumption of simple connectivity is often made in the literature (notably \cite{GreenTao2010}, \cite{GreenTao2012}, \cite{GreenTaoZiegler-2012}).

On the other hand, the definition of \nilbohr{d} sets in \cite[Def. 2.2]{HostKra2009} makes no assumptions about connectivity whatsoever. The purpose of this section is to bridge this gap. We record that the assumption of simple connectivity can be freely added, and that it does not matter whether we work with genuine polynomial sequences or ``linear sequences''.

For a nilmanifold $G/\Gamma$ (with no topological assumptions on $G$), open $U\subset G/\Gamma$ and (polynomial) sequence $g\colon \NN_0 \to G/\Gamma$, denote the corresponding (\nilbohrz{d}) set 
	\begin{equation}
		\label{MAIN:eq:def-of-NilBohr}
		B(g,U) := \set{n \in \NN_0}{g(n)\Gamma \in U }.
	\end{equation}

\begin{lemma}\label{lem:reduction-to-simply-connected}
	Fix $d \geq 1$. Then, for a set $A \subset \NN$, the following properties are equivalent to $A$ being a \nilbohrz{d} set:
\begin{enumerate}
\item\label{cond@reduction:lin} there exists a $d$-step nilmanifold $G/\Gamma$ (with $G$ not necessarily simply connected), an open neighbourhood $e\Gamma \in U \subset G/\Gamma$, and a ``linear'' sequence $g(n) = a^n$ ($a \in G$) such that $A \supset B(g, U)$;

\item\label{cond@reduction:lincon} there exists a $d$-step nilmanifold $G/\Gamma$ with $G$ simply connected (in particular connected), an open neighbourhood $e\Gamma \in U \subset G/\Gamma$, and a ``linear'' sequence $g(n) = a^n$ ($a \in G$) such that $A \supset B(g, U)$;

\item\label{cond@reduction:poly-con} there exists a length $d$ nilmanifold $G_\bullet/\Gamma$ (with each $G_i$ simply connected), an open neighbourhood $e\Gamma \in U \subset G/\Gamma$, and $g \in \poly(\NN \to G_\bullet/\Gamma)$, $g(0) = e$, such that $A \supset B(g, U)$;
\end{enumerate}	 
\end{lemma}
\begin{proof}
	Condition (\ref{cond@reduction:lin}) is the one used to define \nilbohrz{d} sets in Section \ref{sec:INT}.
		
	The equivalence between (\ref{cond@reduction:lincon}) and (\ref{cond@reduction:poly-con}) follows from  Proposition C.2 \cite{GreenTaoZiegler-2012}, which asserts that general polynomial sequences can be lifted to linear ones. Similar argument (without assumption of simple connectivity) appears also in \cite[Proposition 3.14]{Leibman2005} and in the proof of Theorem $\mathrm{B}^*$ in \cite{Leibman-2005b}. Equivalence between (\ref{cond@reduction:lin}) and (\ref{cond@reduction:lincon}) is established in and \cite[Section 1.11]{Leibman2005} by means of embedding an arbitrary nilmanifold in a simply connected one; see also \cite[Section 3.1.2]{HuangShaoYe2014}
\end{proof}

In light of the above considerations, we can --- and will --- always assume that the groups $G_i$ constituting the prefiltration are connected and simply connected.

\begin{remark}\label{rmk:base-point-is-wlog-e}
	Yet another definition of a \nilbohrz{d} set is possible. Recall that the typical \nilbohrz{d} set takes the form $\set{n \in \NN_0}{g(n)\Gamma \in V \Gamma}$, where $e \in V \subset G$ is open, and $g \in \poly(\NN_0 \to G_\bullet)$ with $g(0) = e$. Instead of the assumptions $e \in V$ and $g(0) = e$, we could put a milder restriction $g(0) \in V \Gamma$. Again, this does not lead to a more general notion. As noted in \cite[Section 3.1.4.]{HuangShaoYe2014}, upon replacing $\Gamma$ with $\Gamma' = g(0)\Gamma g(0)^{-1}$, $V$ with $V' = V g(0)^{-1}$, and $g$ with $g'(n) = g(n)g(0)^{-1}$, we recover $\set{n \in \NN}{g(n) \in V \Gamma} = \set{n \in \NN_0}{g'(n) \in V' \Gamma'}$, and $g'(0) = e$. In fact, the original definition in \cite{HostKra2009} uses this seemingly more general form.
\end{remark}

\subsection{VIP-systems}\label{sec:DEF:VIP-systems}

A simple but already interesting instance of the above definitions is the abelian one, where $G = \RR^m$, $\Gamma = \ZZ^m$ and the filtration is given by $$G_0 = G_1 = \dots = G_d = \RR^m,\ G_{d+1} = G_{d+2} = \dots = \{0\}.$$ In this situation, we will simply speak of a polynomial of degree $d$ and keep the filtration implicit. We denote the set of all polynomial maps from $\cA$ to $\TT^m = \RR^m/\ZZ^m$ by $\poly(\cA \to \TT^m)$. 

It is a general fact that if $G_i \supset G_i'$ are two prefiltrations then $\poly(\cA \to G_\bullet) \supset \poly(\cA \to G_\bullet')$. In particular, any polynomial $\cA \to \TT^m$ with respect to some prefiltration of length $d$ is a polynomial of degree $d$, as above.

Such maps are a special case of a more general notion of a \emph{$\mathrm{VIP}$-system}, introduced in \cite{BergelsonFurstenbergMcCutcheon1996}, well before the theory of polynomial maps between nilpotent groups flourished. (In our notation, a {$\mathrm{VIP}$-system} is essentially a polynomial map $\cFe \to \Omega$, where $\Omega$ is an abelian group.)

The following structural result is well known, see \cite[Proposition 2.5]{McCutcheon1999}. We provide a proof for the convenience of the reader; similar ideas will appear in the proof of Proposition \ref{prop:structure-of-poly(F->T)}.

\begin{proposition}[Structure of polynomials $\cF \to \TT^m$]
\label{prop:structure-of-poly(F->T)}

	Suppose that $\nobar{f} \colon \cFe \to \TT^m$ is a polynomial of degree $d$. Then $\nobar{f}$ admits a representation of the form:
	\begin{equation}	
		\nobar{f}(\a) = \sum_{\substack{ \c \subset \a \\ \abs{\c} \leq d}} a_{\c},
		\label{DEF:eq:form-of-poly(F->T)}
	\end{equation}
	where $\a_{\c} \in \TT^m$ are constants. Conversely, any map of the form \eqref{DEF:eq:form-of-poly(F->T)} is a polynomial of degree $\leq d$.
\end{proposition}
\begin{proof}
	Suppose first that $\nobar{f}$ takes the form \eqref{DEF:eq:form-of-poly(F->T)}. We show by induction on $d$ that $\nobar{f}$ is a polynomial of degree $d$; the case $d=0$ is clear. For $d \geq 1$, the discrete difference relation $D_{\b} \nobar{f} = \nobar{f}(\a \cup \b) - \nobar{f}(\a)$ of \eqref{DEF:eq:def-of-D-x} is satisfied for 
	$$
		D_\b \nobar{f} (\a) := 
		\sum_{\substack{\gamma \subset \a \\ \abs{\gamma} \leq d-1}} { 
		\sum_{\substack{\delta \subset \b \\ \abs{\delta} \leq d-\abs{\gamma} } }
		a_{\gamma \cup \delta} },
	$$
which is again of the form \eqref{DEF:eq:form-of-poly(F->T)} with degree $d-1$, so the claim follows.

Conversely, let $\nobar{f}\colon \cFe \to \TT^m$ be a polynomial of degree $\leq d$. By a inclusion-exclusion argument, it is easy to construct $\nobar{g}$ of the form \eqref{DEF:eq:form-of-poly(F->T)} such that $\nobar{f}(\a) = \nobar{g}(\a)$ if $\abs{\a} \leq d$. Hence, replacing $\nobar{f}$ with $\nobar{f} - \nobar{g}$ if necessary, we may assume that $\nobar{f}(\a) = 0$ whenever $\abs{\a} \leq d$. We show, by induction on $d$ this condition implies $\nobar{f}(\a) = 0$ for all $\a \in \cFe$. Because degree $0$ polynomials are constant, the case $d=0$ is clear; assume $d \geq 1$.

Choose any $\b \in \cF$ with $\abs{\b} = 1$. By the Definition \ref{def:poly(A->G)}, there is a polynomial $D_\b \nobar{f}$ of degree $d-1$ such that 
$$
D_\b \nobar{f} (\a) = \nobar{f} (\a \cup \b) - \nobar{f}(\a), \qquad (\a \cap \b = \emptyset).
$$
If $\abs{\a} \leq d-1$ then $D_\b \nobar{f}(\a) = 0$. Using the inductive assumption, we conclude that $D_\b \nobar{f}(\a) = 0$ for all $\a \in \cFe$, $\a \cap \b = \emptyset$. Because $\b$ was arbitrary, if $\nobar{f}(\a) = 0$ for all $\a \in \cFe$ of a given size $\abs{\a} = n$, then the same holds for $\abs{\a} = n+1$; hence by induction $\nobar{f}(\a) = 0$ for all $\a$.
\end{proof}

\subsection{Host-Kra cube groups}\label{sec:DEF:HostKra}
\renewcommand{\k}{k}
\newcommand{\Q}[1]{ {\{0,1\}^{#1}} }
A useful approach to polynomial maps is obtained by introducing the notion of \emph{Host-Kra cube}, or more generally \emph{cube group} $\HK^{\k}(G_{\bullet})$. This notion was first introduced (thought with a different name) by Host and Kra in \cite{HostKra2005}, and is extensively used in a number of papers, including \cite{GreenTao2010, GreenTao2012}. It is also a basis for the work of Szegedy and Camarena
\cite{SzegedyCamarena2010} later refined by Gutman, Manners and Varj\'{u} \cite{GutmanMannersVarju-1,GutmanMannersVarju-2,GutmanMannersVarju-3}. We use some basic facts, using \cite[Appendix E]{GreenTao2010} and \cite{GreenTao2012} as our main reference; an accessible introduction can be found in \cite{Green-book}. Throughout, $k \geq 1$ is an integer, and $G_\bullet$ is a prefiltration consisting of simply connected groups.

For any $\k$, we may consider the \emph{cube} $\Q\k$. It is often convenient to identify $\Q\k$ with the powerset $\cP([\k])$. In particular, $\Q\k$ carries a natural partial order where $\omega \leq \omega'$ if $\omega_i \leq \omega_i'$ for all $i$.

For $\omega \in \Q\k$ and $g \in G$, we define $g^{[\omega]} \in G^{\Q\k}$ by
$$g^{[\omega]}_{\sigma} =
\begin{cases}
g & \text{if } \sigma \geq \omega, \\
e_G & \text{otherwise}.
\end{cases}$$
Hence, $g^{[\omega]}$ can be viewed as a cube with entries $g$ on the upper face $\set{\sigma \in \Q\k }{ \sigma \geq \omega }$, and $e_G$ elsewhere. 

If $G_\bullet$ is a prefiltration, we further define the \emph{face group} $G^{[\omega]}$ to be the subgroup of $G^{\Q\k}$ generated by elements of the form $g^{[\omega]}$ with $g \in G_{\abs{\omega}}$, where $\abs{\omega} := \abs{\set{i \in [k] }{ \omega_i = 1}}$. Finally, we define the Host-Kra cube group $\HK^{\k}(G_{\bullet})$ to be the group generated by all face groups $G^{[\omega]}$. 

Faces mentioned above can be thought of as ``upper faces''. One can, for a face $F = \set{\omega \in \Q{k}}{\omega_i = \sigma_i \text{ for } i \in I}$ of codimension $\codim F = \abs{I}$, consider the face group $G^{[F]}$ generated by elements $g^{[F]}$ given by $g^{[F]}_{\sigma} =
\begin{cases}
g & \text{if } \sigma \in F, \\
e_G & \text{otherwise},
\end{cases}$
where $g \in G_{\abs{I}}$. These are, however, contained in the group $\HK^{\k}(G_{\bullet})$ defined above (see \cite{GreenTao2010}, discussion after E.5).

A basic fact lying at the basis for many inductive arguments is the commutator relation:
\begin{equation}
[G^{[\sigma]}, G^{[\rho]}] \subset G^{[\sigma \cup \rho]}, \qquad \rho, \sigma \in \Q{k},
\label{fact:HK:commute}
\end{equation}
which follows directly from the identity $[g^{[\sigma]}, h^{[\rho]}] = [g,h]^{[\sigma \cup \rho]}$ (see \cite[Lemma E.5]{GreenTao2010}).

For a (total) order $\prec$ on $\Q{k}$, compatible with inclusion in the sense $\sigma \leq \rho$ coordinatewise, then also $\sigma \preceq \rho$ ($\sigma,\rho \in \Q{k}$), we define ordered products $\prod_{\omega \in \Omega}^\prec \bb x_\omega := \bb x_{\omega_1} \bb x_{\omega_2} \dots \bb x_{\omega_r}$, where $\Omega = \{\omega_1 \prec \omega_2 \prec \dots \prec \omega_r\} \subset \Q{k}$ (we use this with $\bb x_\omega \in G^{\Q{k}}$). 
	
\begin{fact}\label{fact:HK:representation}
If $\bb g = (g_\omega) \in \HK^k(G_\bullet)$, then the representation $\bb g = \prod_{\omega }^{\prec} \tilde g_{\omega}^{[\omega]}$ with $\tilde g_\omega \in G_{\abs{\omega}}$ exists and is unique. Moreover, $\tilde g_\omega$ is a word in $ g_{\sigma}$,  $\sigma \leq \omega$.
\end{fact}	
\begin{proof}
	See \cite[Lemma E.6]{GreenTao2010}, or \cite[Lemma 6.4]{GreenTao2012}.
\end{proof}

The key reason for interest in the Host-Kra cube group is the characterisation of polynomial maps which they provide. Let $\cFe^{[\k]}$ denote the set of parallelepipeds, i.e.\ cubes of the form $\{\a_\omega\}_{\omega \in \Q\k}$ with $\a_\omega = \a_0 \cup \bigcup_{i \in \omega} \a_i$ for some disjoint $\a_0,\a_1,\dots, \a_k \in \cFe$. 

\begin{proposition}\label{prop:poly-FCAE-G}
Let $G_{\bullet}$ be a length $d$ prefiltration.
	A sequence $f \colon \cFe \to G$ is polynomial in the sense of Definition \ref{def:poly(A->G)} if and only if $f$ maps $\cFe^{[\k]}$ to $\HK^\k(G_\bullet)$ for each $\k$.
\end{proposition}

Note that because $\HK^{k}(G_\bullet)$ are groups, Theorem \ref{thm:poly-is-group} follows from Proposition \ref{prop:poly-FCAE-G} immediately.

To prove this proposition, it is convenient to isolate another fact on Host-Kra cube groups (implicit in \cite[Proposition 6.5]{GreenTao2012}). For $\bb g, \bb g' \in G^{\Q{k}}$, denote by $\bb h = (\bb g, \bb g') \in \G^{\Q{k+1}}$ the result of ``glueing'' $\bb g$ and $\bb g'$, i.e.\ the cube with $\bb h_{\omega 0} = \bb g_{\omega}$ and $\bb h_{\omega 1} = \bb g'_{\omega}$.
\begin{fact}\label{fact:HK:cuts}
	For a cube $\bb g \in G^{\Q{k}}$ the following are equivalent:
	\begin{enumerate}[label=({\arabic*}a),ref=({\arabic*}a)]
	\item\label{cond:HK:1a} $\bb g \in \HK^{k}(G_\bullet)$;
	\item\label{cond:HK:2a} $(\bb g, \bb g) \in \HK^{k+1}(G_\bullet)$.
	\end{enumerate}
	Also, the following are equivalent:
	\begin{enumerate}[label=({\arabic*}b),ref=({\arabic*}b)]
	\item\label{cond:HK:1b} $\bb g \in \HK^{k}(G_{\bullet+1})$;
	\item\label{cond:HK:2b}  $(\bb e, \bb g) \in \HK^{k+1}(G_{\bullet+1})$, where $\bb e = (e_G)_{\omega \in \Q{k}}$.	
	\end{enumerate}
\end{fact}
\begin{proof}
	Fix a compatible order $\prec$ on $\Q{k}$. For the implication \ref{cond:HK:1a} $\Rightarrow$ \ref{cond:HK:2a}, write (using Fact \ref{fact:HK:representation}) $\bb g = \prod_{\omega \in \Q{k}}^{\prec} g_{\omega}^{[\omega]}$. Then $(\bb g, \bb g) = \prod_{\omega \in \Q{k}}^{\prec} g_{\omega}^{[\omega0]} \in \HK^{k+1}(G_\bullet)$ if $\bb g \in \HK^{k}(G_\bullet)$. By the same token, for \ref{cond:HK:1b} $\Rightarrow$ \ref{cond:HK:2b}, we have $(\bb e, \bb g) = \prod_{\omega \in \Q{k}}^{\prec} g_{\omega}^{[\omega1]} \in \HK^{k+1}(G_\bullet)$ if $\bb g \in \HK^{k}(G_{\bullet+1})$.
	
	Conversely, for \ref{cond:HK:2a} $\Rightarrow$ \ref{cond:HK:1a}, if $(\bb g, \bb g) = \prod_{\omega \in \Q{k+1}}^{\prec} g_\omega^{[\omega]} \in \HK^{k+1}(G_\bullet)$, then $\bb g =  \prod_{\omega \in \Q{k}}^{\prec} g_{\omega 0}^{[\omega]} \in \HK^{k}(G_\bullet)$. Finally, to prove \ref{cond:HK:2b} $\Rightarrow$ \ref{cond:HK:1b}, write $(\bb e, \bb g) = \prod_{\omega \in \Q{k+1}}^{\prec} g_\omega^{[\omega]} \in \HK^{k+1}(G_\bullet)$. Because of uniqueness in Fact \ref{fact:HK:representation}, we have $g_{\omega 0} = e$ for $\omega \in \Q{k}$, so $\bb g =  \prod_{\omega \in \Q{k}}^{\prec} g_{\omega 1}^{[\omega]} \in \HK^{k}(G_{\bullet+1})$.
\end{proof}

\begin{proof}[Proof of Proposition \ref{prop:poly-FCAE-G}]
This follows by a standard modification of the proof of \cite[Proposition 6.5]{GreenTao2012}. The key point is that for each $k \geq 1$ and $(\a_\omega)_\omega \in \cFe^{[k]}$, the cube $\bb g = (f(\a_\omega))_{\omega \in \Q{k}}$ can be written as the product $\bb g = (\bb g', \bb g') (\bb e, \bb h)$, where $\bb g' = (f(\a_\omega))_{\omega \in \Q{k-1}}$ and $\bb h = (D_{\a_k} f(\a_\omega))_{\omega \in \Q{k-1}}$.

Suppose that $f \in \poly(\cFe \to G_\bullet)$, and proceed by induction on $k$ and $d$. We have $\bb g' \in \HK^{k-1}(G_\bullet)$ and $\bb h \in \HK^{k-1}(G_{\bullet+1})$, with notation as above. Hence, by \ref{fact:HK:cuts}, $(\bb g', \bb g'), (\bb e, \bb h) \in \HK^{k}(G_\bullet)$, so $\bb g \in \HK^{k}(G_\bullet)$.

Suppose conversely that $f$ maps $\cFe^{[k]} \to \HK^{k}(G_\bullet)$ for all $k$, and proceed by induction on $d$. Then, again with notation as above, $\bb g' \in \HK^{k-1}(G_\bullet)$, so $(\bb g', \bb g') \in \HK^{k}(G_\bullet)$. Hence, $(\bb e, \bb h) \in \HK^{k}(G_\bullet)$ and by another application of \ref{fact:HK:cuts}, $\bb h \in \HK^{k}(G_{\bullet+1})$. Since $D_{\a_k}f$ maps $\cFe^{[\k]}$ to $\HK^\k(G_{\bullet+1})$, by inductive assumption $D_{\a_k} f \in \poly(\cFe \to G_{\bullet+1})$. Since $\a_k$ was arbitrary, it follows that $f \in \poly(\cFe \to G_\bullet)$.
\end{proof}

\newcommand{\Qd}[1]{{ \{0,1\}^{#1}_{*} }}

\subsection{Host-Kra cubes and nilmanifolds} To deal with polynomial maps $\poly(\cFe \to G_\bullet/\Gamma)$, we introduce the following nilmanifold analogue of the Host-Kra group. For each $k$, we define $\HK^\k(G_\bullet/\Gamma) \subset (G/\Gamma)^{\Q{k}}$ to be the natural projection of $\HK^\k(G_\bullet)$. Note that here we depart from the definitions in \cite{GreenTao2010} (compare Definition E.8).

\begin{lemma} \label{lem:compactness-of-HK(G/Gamma)}
	If $G_\bullet/\Gamma$ is a nilmanifold, then for each $k$ the space $\HK^\k(G_\bullet/\Gamma)$ is compact.
\end{lemma}
\begin{proof}
By \cite[Lemma E.10]{GreenTao2010}, $\Gamma^{\Q{k}} \cap \HK^{k}(G_\bullet)$ is cocompact and discrete in $\HK^{k}(G_\bullet)$. Because the projection  $ \HK^{k}(G_\bullet) \to \HK^\k(G_\bullet/\Gamma)$ factors through the quotient $ \HK^{k}(G_\bullet) / \bra{ \Gamma^{\Q{k}} \cap \HK^{k}(G_\bullet)}$, the space $\HK^\k(G_\bullet/\Gamma)$ is compact.
\end{proof}

The Host-Kra cube group has a ``corner completion property'', stating that it is always possible to complete a partial cube with a single missing entry. Here, by $\Qd\k$ we denote the cube with missing upper corner $\Q\k \setminus \{1^{\k}\}$. Note that any face, i.e.\ a set of the form $\set{ \omega \in \Q\k }{ \omega_i = \sigma_i \text{ for } i \in I }$, can be naturally be identified with $\Q{\k - \abs{I}}$. 
Hence, if $F$ is a face then for $(g_\omega)_{\omega \in F} \subset G^{F}$ it makes sense to ask if $ (g_\omega)_\omega \in \HK^{k-\abs{I}}(G_\bullet)$ (and likewise for $G/\Gamma$ in place of $G$).

\begin{lemma}\label{lem:corner-completion-Gamma}
Let $G_{\bullet}/\Gamma$ be a length $d$ nilmanifold.
Suppose that $g_\omega \in G$, $\omega \in \Qd{k}$ are such that for any codimension $1$ face $F$ with $1^k \not \in F$, the restricted cube $(g_{\omega})_{\omega \in F}$ lies in $\HK^{[k-1]}(G_\bullet)$. Then, there exists $g_{1^k} \in G$, which completest the cube: $(g_{\omega})_{\omega \in \Q{k}} \in \HK^{k}(G_\bullet)$. 
Moreover, if $g_\omega \in \Gamma$ for all $\omega \in \Qd{k}$, then it is possible to choose $g_{1^k} \in \Gamma$.
\end{lemma}
\begin{proof}
This is essentially \cite[Lemma E.7]{GreenTao2010}.	Fix a compatible order $\prec$ of $\Q{k}$. For each lower codimension $1$ face $F = \set{\omega \in \Q{k}}{\omega_j = 0}$ (for some $j \in [k]$) we have by \ref{fact:HK:representation} a representation $(g_\omega)_{\omega \in F} = \prod^{\prec}_{\omega \in F} \tilde g_\omega^{[\omega]} |_{F}$. Because of uniqueness in   \ref{fact:HK:representation}, the coefficients $\tilde g_\omega^{[\omega]}$ do not depend of $F$. Hence, we may define $\bb g = \prod^{\prec}_{\omega \in \Q{k}} \tilde g_\omega^{[\omega]}$; this is a completion of  $(g_\omega)_{\omega \in \Qd{k}}$.
	
	For the additional part, note that if $g_\omega \in \Gamma$ for all $\omega \in \Qd{k}$, then also $\tilde g_\omega \in \Gamma$. Hence, with the construction above, $g_{1^k} = \prod_{\omega}^\prec \tilde g_{\omega} \in \Gamma$.	
\end{proof}

We are now ready to prove the analogue of Proposition \ref{prop:poly-FCAE-G} for nilmanifolds.

\begin{proposition}\label{prop:poly-FCAE}
	Let $G_{\bullet}/\Gamma$ be length $d$ nilmanifold.
	A sequence $\bar f \colon \cFe \to G/\Gamma$ is polynomial in the sense of Definition \ref{def:poly(A->G/Gamma)} if and only if $\bar f$ maps $\cFe^{[\k]}$ to $\HK^\k(G_\bullet/\Gamma)$ for each $\k$.
\end{proposition}
\begin{proof}

	If $\bar f = \pi \circ f$ is a polynomial in the sense of Definition \ref{def:poly(A->G/Gamma)}, then $f$ maps $\cFe^{[\k]}$ to $\HK^{\k}(G_\bullet)$, and hence $\bar f$ maps $\cFe^{[\k]}$ to the projected image, $\HK^k(G_\bullet/\Gamma)$.
	
	Suppose conversely that $\bar f$ maps $\cFe^{[\k]}$ to $\HK^{\k}(G_\bullet/\Gamma)$ for each $\k$. We aim to construct $f \in \poly(\cFe \to G_\bullet)$ such that $\bar f = \pi \circ f$. Our construction will be inductive, where $f(\a)$ is constructed only after all $f(\b)$ with $\abs{\b} < \abs{\a}$. To initiate the construction, pick arbitrary $f(\emptyset)$ with $\pi(f(\emptyset)) = \bar{f}(\emptyset)$.
	
	Suppose we want to assign a value to $f(\a)$.  Let $\a = \{a_1,a_2,\dots,a_k\}$, and write $\a_\omega = \set{ a_i }{ \omega_i = 1}$, so that $(\a_{\omega})_{\omega}$ defines a parallelepiped. Hence, we have the cube $\bar {\bb g} = (\bar{f} (\a_{\omega}))_{\omega} \in \HK^\k(G_\bullet/\Gamma)$. By definition, this cube is the projection of some cube $\bb g = (g_\omega)_{\omega} \in \HK^\k(G_\bullet)$. On the other hand, because $f(\b)$ have been constructed for $\beta \subsetneq \a$, we have the partial cube $({f} (\a_{\omega}))_{\omega \neq 1^\k}$. 
	
	Consider the (partial) cube $(g_{\omega}^{-1} f(\a_\omega))_{\omega \in \Qd{k}}$. By construction, $g_{\omega}^{-1} f(\a_\omega) \in \Gamma$ for all $\omega \in \Qd{k}$. Thus, by Lemma \ref{lem:corner-completion-Gamma}, there is some $\gamma \in \Gamma$ which completes this partial cube; that is if we define $f(\a) = f(\a_{1^k}) = g_{1^k} \gamma$, then $(g_{\omega}^{-1} f(\a_\omega))_{\omega \in \Q{k}} \in \HK^{k}(G_\bullet)$. Multiplying by $\bb g$, we conclude that $(f(\a_\omega))_{\omega \in \Q{k}} \in \HK^{k}(G_\bullet)$. 
	
This construction guarantees that $f$ maps parallelepipeds of the special form $\a_\omega = \set{ a_i }{ \omega_i = 1 }$ into the Host-Kra cube groups. It remains to see that any parallelepiped can be suitable embedded into one of this special form.

	Indeed, let $(\a_\omega)_{\omega \in \Q{k}}$ be any parallelepiped, where $\a_\omega = \a_0 \cup _{i \in \omega} \a_i$ with $\a_i$ disjoint. Let $\{a_1,\dots,a_l\}$ be all the elements of $\bigcup_{i=0}^k \a_i$, and consider the projection map $P\colon G^{\Q{l}} \to G^{\Q{k}}$, mapping $\bb g = (g_\omega)_\omega$ to $\bb h = (h_\sigma)_\sigma$ with $h_{\sigma} = g_{\alpha_\sigma}$. Once we show that $P$ maps $\HK^{l}(G_\bullet)$ to $\HK^{k}(G_\bullet)$, it will follow that $f$ maps $\cFe^{[k]}$ to $\HK^{k}(G_\bullet)$. Because $P$ preserves multiplication, it suffices to verify that it maps generators of $\HK^{l}(G_\bullet)$ to $\HK^{k}(G_\bullet)$. This follows from the observation that $P(g^{[\omega]}) = g^{[\sigma]}$, where $\sigma = \set{i \in [k]}{\a_i \cap \omega \neq \emptyset}$ has size $\leq \abs{\omega}$, and thus $P$ maps $G^{[\omega]}$ to $G^{[\sigma]} \subset \HK^{k}(G_\bullet)$.
\end{proof}

\begin{corollary}
	Theorem \ref{thm:poly-is-compact} holds.
\end{corollary}
\begin{proof}
	The set $\poly(\cFe \to G_\bullet/\Gamma)$ is lies in the compact space $(G/\Gamma)^{\cFe}$, so it will suffice to check that it is closed.

	Suppose that $\bar f_n \to \bar f$ pointwise. Then for each $\k$, and for each parallelepiped $(\a_{\omega})_{\omega \in \Q\k} \in \cFe^{[k]}$, the cube $( \bar f (\a_\omega))_\omega$ lies in the closure of $\HK^\k(G_\bullet/\Gamma)$. But $\HK^\k(G_\bullet/\Gamma)$ is already closed by Lemma \ref{lem:compactness-of-HK(G/Gamma)}, so we $\bar{f}$ maps $\cFe^{[k]}$ to $\HK^\k(G_\bullet/\Gamma)$, so we are done.
\end{proof}

\section{$\cS_k$-sequences}\label{sec:SEQ}

In this section, we develop some language to speak about sequences indexed by $\cS_k$, the $k$-syndetic sets. The key insight here is that $\cS_k$-indexed sequences admit a well behaved notion of a subsequence, which allows us to restrict to particularly structured sequences later on, through application of Proposition \ref{prop:structure-theorem-stable}.

\subsection{IP sets revisited} 
We briefly return to the discussion of \ip\ sets, which serve as a motivation and an analogue for the \sg{k} sets. Recall that an \ip\ set is (a superset of) a set of the form $\FS(n_i) = \set{ n_\a }{ \a \in \cF}$ for some $n_i \in \NN$, where $n_\a = \sum_{i \in \a} n_i$. 

One of the reasons for interest in the \ip\ sets is the celebrated theorem of Hindman \cite{Hindman1974}, stating that the class of \ip\ sets is \emph{partition regular}. Here, we say that a class of sets $\cC \subset \cP(\NN)$ is \emph{partition regular} if for any $A \in \cC$ and any finite partition $A = \bigcup_{i = 1}^k A_i$ there exists some $i \in [k]$ such that $A_i \in \cC$. 

Barring spurious coincidences between different terms $n_\a$ and $n_\b$ with $\a \neq \b$, the set $\FS(n_i)$ can for most intents and purposes be identified with the set $\cF$ of finite subsets of $\NN$. To make this idea more precise, recall that we endow $\cF$ with the structure of a partial semigroup, where the operation is the disjoint union. Now, the map $\a \mapsto n_\a$ is a morphism of partial semigroups. Moreover, any morphism of partial semigroups $\cF \to \NN$ takes the form $\alpha \mapsto n_\alpha = \sum_{i \in \alpha} n_i$ for some sequence $n_i$, and the \ip\ sets are precisely the images of $\cF$ in $\NN$ by such morphisms. 

This point of view makes it more convenient to speak of \ip\ subsets of a given \ip\ set. If $(\a_i)_{i=1}^\infty$ is a sequence of disjoint sets, then the map $\b \mapsto \a_\b := \bigcup_{i \in \b} \a_i$ is a morphism of partial semigroups, which is in fact an isomorphism onto the image. We will call the image of such a morphism an \emph{\ip\ ring}\footnote{Note that we only require $\a_i$ to be pairwise disjoint. A similar definition is sometimes made with a stronger condition $\max \a_i < \min \a_{i+1}$. We do not follow this approach here.} and denote it by $\FU(\a_i)$. If $f$ is an $\cF$-indexed sequence and $\FU(\a_i)$ is an \ip\ ring, it is natural to consider the ``restriction'' of $f$ to $\FU(\a_i)$, given by $\tilde f(\b) = f(\a_\b)$,where as usual $\a_\b = \bigcup_{i \in \b} \a_i$. (Note that the definition is arranged so that the domain of $\tilde f$ is again $\cF$.)

We can now reformulate Hindman's theorem as follows. Suppose that a sequence $f\colon \cF \to X$ is given, taking values in some finite set $X$. Hindman's theorem then ensures that for a suitable choice of the \ip\ ring, the corresponding subsequence $\tilde f$ is constant. Using the succinct terminology of Zorin-Kranich \cite{Zorin-Kranich2013}, any finitely valued $\cF$-sequence is \emph{wlog} constant.

Slightly more generally, Hindman's theorem is equivalent to the statement that for any sequence $f\colon \cF \to X$ taking values in a compact space $X$, there exists an \ip\ ring $\FU(\a_i)$ such that the limit $\plim{\operatorname{IP}}{\b} f(\a_\b)$ exists. (We do not define $\plim{\operatorname{IP}}{}$ here, but see e.g.\ \cite[Sec. 1]{BergelsonFurstenbergMcCutcheon1996} for details.) While this result is never directly used, nor even properly stated, in this paper, it serves as a motivation for Proposition \ref{prop:structure-theorem-stable}, which plays a key role.

\subsection{Basic definitions}\label{sec:SEQ:Definitions}
\label{sec:SEQ:Subsequences}

We wish to adapt some of the ideas relevant to \ip\ sets to the context of \sg{k} sets.

As suggested by the formulation of Question \ref{question:basic}, we will be interested in polynomial sequences $\poly(\cFe \to G_{\bullet}/\Gamma)$, but only values at $\cS_{k}$, the $k$-syndetic sets, will play a role. To emphasise this state of affairs, we will use the term \emph{ $\cS_{k}$-sequence} to refer to a $\cFe$-indexed sequence which we only intend to evaluate on $\cS_{k}$. If we consider a $\cFe$-indexed sequence with no $\cS_{k}$ in mind, we refer to it  as an $\cFe$-sequence, or simply a sequence.

\begin{remark}\label{remark:Sk-not-enough}
It is tempting to dispose of $\cFe$ altogether, and work with sequences indexed by $\cS_k$. Indeed, $\cS_{k}$ is certainly a partial semigroup, so it makes sense to consider polynomial groups such as $\poly(\cS_k \to G_{\bullet})$, and much of the discussion in this section would carry through with minor modifications.

However, we pursue a different route. One of the reason is that there does not seem to be a satisfactory analogue of Proposition \ref{prop:structure-of-poly(F->T)} for $\poly(\cS_{k} \to \TT)$. Indeed, the algebraic structure of $\cS_{k}$ is not strong enough to admit a good description of the polynomials from $\cS_{k}$. For instance, $\cS_k$ has no non-degenerate parallelepipeds of dimension $>k+1$.
\end{remark}

\subsection{Subsequences}

We will now introduce a notion of a subsequence suitable for the study of $\cS_k$-sequences. Consider a sequence $f\colon \cFe \to X$, taking values in some space $X$. As a source of motivation, recall that for any \ip\ ring $\FU(\a_i)$, we may construct a subsequence $\tilde f$ given by $\tilde f(\b) = f(\a_\b)$, where $\a_\b = \bigcup_{i \in \b} \a_i$. Suppose now that $\a_j$ are chosen so that $\a_\b$ is a $\cS_k$ for any $\b$ in $\cS_k$. Then, almost tautologically, we have $\tilde f(\cS_k) \subset f(\cS_k)$, so whenever we are interested in proving statements such as ``there exists $\a \in \cS_k$ such that $f(\a) \in U$'', we may equally well replace $f$ with $\tilde f$. 

The condition that $\a_\b$ should be $\cS_k$ whenever $\b$ is $\cS_k$ will be satisfied in the case when $\a_j$ takes the form $\a_j = \{i_j, i_j+{k}, \dots, i_{j+{k}} -{k} \}$ for an increasing sequence $(i_j)_{j =1}^\infty$ with $i_{j+k} \equiv i_j \pmod{k}$. Indeed, it suffices to check that if $j < j' \leq j+k$, then the gaps of $\a_{j} \cup \a_{j'}$ are bounded by $k$. If $j = j+k$, then $\a_{j} \cup \a_{j'}$ is just a progression with step $k$, and if $j' < j+k$ then $\a_{j} \cup \a_{j'}$ consists of two overlapping progressions with step $k$ --- in either case, the bound on the gaps is clear.

We will show that any possible choice of $\a_j$ such that $\b \mapsto \a_\b$ preserves $\cS_k$ will essentially be of the above form, with the inconsequential caveat that there is some additional freedom in the choice of $\a_1,\a_2, \dots, \a_{k+1}$. Although formally not necessary, we record the proof of this fact, lest the definition which follows be unmotivated.

\renewcommand{\i}{\iota}
\begin{lemma}\label{SEQ:obs:form-of-Sk-sseq}
	Let $k \in \NN$. Suppose that $\FU(\a_i)$ is an \ip\ ring such that $\a_\b \in \cS_k$ for any $\b \in \cS_k$. Then, there exists an increasing sequence $(i_j)_{j=1}^\infty$ with $i_{j+k} \equiv i_{j} \pmod{k}$ for all $j$, such that $\a_j = \{i_j, i_j+{k}, \dots, i_{j+{k}} -{k} \}$ for any $j > k+1$.
\end{lemma}
\begin{proof}
	The basic idea is to iterate over $\NN$ in increasing order and show that if $\a_j$ were not of the aforementioned form, then we would encounter a gap $>k$ in $\a_\b$ for some $\b \in \cS_k$.

	Let $i_*$ be such that $\max \a_{i_*}$ is minimal. For $t \geq t_* := \max \a_{i_*} + k + 1$, consider the set $\iota(t)$ of indices of $i$ such that $\a_i \cap [t-k,t) \neq \emptyset$. In particular, if $\min \a_i < t$ and $i \not \in \i(t)$, then $t \not \in \a_i$ because $\a_i \in \cS_k$. 
	
	We claim that for each $t \geq t_*$, $\i(t)$ is a connected interval of length $k$. Indeed, let $n$ be sufficiently large that $\min \a_n \geq t$, and consider the set $\b = [n] \setminus \iota(t)$. Then $\min \a_\b \leq \min \a_{i_*} < t-k$ and $\max \a_{\b} \geq \max \a_{n} > t$ and $[t-k,t)\cap \a_\b = \emptyset$, so $\a_\b \not \in \cS_k$. Hence, $\b \not \in \cS_k$. Because $\abs{\i(t)} \leq k$, the only way for $\b = [n] \setminus \i(t)$ to have gap $> k$ is if $\i(t)$ is a connected interval of length $k$, as claimed.
	
	For each $t$ and $0 \leq l \leq k$, denote by $i_l(t)$ the index such that $t-l \in \a_{i_l(t)}$ (each $t-l$ belongs to some $\a_i$ by the above claim, and $i_1(t),\dots,i_k(t)$ are all distinct). Since both $\i(t)$ and $\i(t+1)$ are intervals of length $k$, one of the following holds: (a) $i_0(t) = i_k(t)$, or (b) $i_k(t) = \min \i(t)$ and $i_0(t) = i_k(t) + k$, or (c) $i_k(t) = \max \i(t)$ and $i_0(t) = i_k(t) - k$. 
	
	We next show that (c) never actually occurs. Clearly, (b) happens infinitely often. Hence, if (c) happens at any point, then there exists $t$ such that (c) happens at $t$ and the next of (b) and (c) to happen is (b). Thus, at some $t'>t$, (b) happens, and for $t < s < t'$, (a) happens. Hence, $\i(t') = \i(t+1) = \i(t) \setminus \{ i_k(t) \} \cup \{i_k(t) - k\}$. Because $\max \i(t') = i_k(t) + k$, we have $i_0(t') = i_k(t)$ --- but this is impossible, since then $\a_{i_k(t)}$ would have gap $> k$.
	
	Note that if (a) happens then $\i(t+1) = \i(t)$, and if (b) happens then $\i(t+1) = \i(t)+1$. Each $i \neq i_*$ appears as the element of some $\i(t)$, and the only way for this to happen is if $i_* = 1$ and $\i(t_*) = \{2,3,\dots,k+1\}$.	Inspecting what happens in each of (a) and (b), we may further conclude that each of the sets $\a_i \cap [t_*, \infty)$ is a progression of step $k$, and for $i > k+1$, $\a_i \subset [t_*, \infty)$ and $\min \a_{i+1} > \min \a_{i}$. It follows that $\a_i$ are of the desired form.
\end{proof}

We are now ready to state the definition of a $\cS_l$-subsequence.

\begin{definition}
	Let $f\colon \cFe \to X$ be a $\cS_k$-sequence, $k \geq 1$. Then, a \emph{$\cS_k$}-subsequence of $f$ is any sequence $\tilde f\colon \cFe \to X$ taking the form $\tilde f(\b) = f(\a_\b)$ where $\a_j = \{i_j, i_j+{k}, \dots, i_{j+{k}} -{k} \}$ for some increasing sequence $i_j$ such that $i_{j+k} \equiv i_j \pmod{k}$, and $\a_\b := \bigcup_{i \in \b} \a_i$.
	
	Similarly, for any $l < k$, $k \geq 1$ or $k = l = 0$, a $\cS_l$-subsequence of a $\cS_k$-sequence $f$ is a sequence $\tilde f\colon \cFe \to X$ taking the form $\tilde f(\b) = f(\a_\b)$ where $\a_i \in \cF$ are pairwise disjoint and the map $\b \mapsto \a_\b := \bigcup_{i \in \b} \a_i$ takes $\cS_l$ to $\cS_k$. Note that this definition depends on both $l$ and $k$.
\end{definition}

\begin{remark}
	There is a slight mismatch in the above definition. Namely, the notion of $\cS_k$-subsequence of a $\cS_k$-sequence is not exactly the same as would be obtained by extrapolating the definition of a $\cS_l$-subsequence $(l < k)$ to  $l=k$.
	
	This could be amended by requiring, in the definition of $\cS_l$-subsequence ($l<k$), not only that $\b \mapsto \a_\b$ maps $\cS_l$ to $\cS_k$, but also that $\a_i$ can be extended to $(\a_i)_{i = - \infty}^{\infty}$, $\a_i \subset \ZZ$, in such a way that $\a_\b$ has gaps $\leq k$ if $\b$ has gaps $\leq l$ for $\b \subset \ZZ$, finite. 
	
	Although this would make the definitions more consistent, they would also become more complicated, so we instead accept the mismatch for the sake of simplicity.
\end{remark}

\begin{example}\label{SEQ:ex:S0}
	As already remarked, $\cS_0$-sequences can be identified with the usual $\NN$-indexed sequences. Under this identification, if $f \colon \cFe \to X$ is a $\cS_0$-sequence, then any subsequence $(f(\{i_j\}))_{j=1}^\infty$ (in the usual sense of the word) can be seen as corresponding to a $\cS_0$-subsequence $\tilde f$ with $\tilde f(\{j\}) = f(\{i_j\})$.
\end{example}

\begin{example}\label{SEQ:ex:rational}
	Let $f \colon \cFe \to \TT$ be a linear sequence of the special form $f(\a) = \abs{\a}/M \bmod{1}$, viewed as a $\cS_k$-sequence, $k \geq 1$. Pick the sequence $i_{l + ak} = l + M a k$ ($l \in [k],\ a \in \NN_0$), so that $\a_j = \{i_j, i_{j}+k, \dots, i_{j+k} - k\}$ are arithmetic progressions of length $M$ and step $k$. Then, $\tilde f(\b) = f(\a_\b) = 0 \in \TT$ is a $\cS_k$-subsequence of $f$. 
	Note that for $k = 1$, this amounts to splitting $\NN$ into intervals of length $M$. 
\end{example}

\begin{example}\label{SEQ:ex:linear}
	Let $f \colon \cFe \to \tt$ be a general linear sequence $f(\a) = \sum_{i \in \alpha} a_i$ for some $a_i \in \TT$. To begin with, view $f$ as a $\cS_1$-sequence. There exists a convergent subsequence of the partial sums: $\sum_{i=1}^{i_j-1} a_i \to c$ for some $i_j \to \infty$, increasing, and $c \in \TT$.
	
	Pick $\a_j = \{i_j, \dots, i_{j+1}-1\}$. Then 
	$$f(\a_j) = \sum_{i \in \a_j} a_i = \sum_{i=1}^{i_{j+1}-1} a_j - \sum_{i=1}^{i_{j}-1} a_j =: \e_j \to 0 \quad \text{as} \quad j \to \infty.$$
	
	We may now construct the $\cS_1$-subsequence $\tilde f$ of $f$ given by $\tilde f (\b) = \sum_{j \in \b} \e_j$. Note that, by suitable choice of $(i_j)_j$, we may ensure that the $\e_j$ are as small as we please. Requiring that $\norm{\e_j} < \delta 2^{-j}$ for some small $\delta > 0$, say, we may thus ensure that $\norm { \tilde f(\b) } < \delta$ for all $\b \in \cFe$.
	
	By a similar construction, it can be checked that if $f$ is considered as a $\cS_k$-sequence, we may again extract a $\cS_k$-subsequence $\tilde{f}$ taking the form $\tilde f(\b) = \sum_{j \in \b} \e_j$, where $\e_j \to 0$ rapidly as $j \to \infty$. Because the details do not include any new ideas, we skip them.
\end{example}

\begin{example}\label{SEQ:ex:S1}
	Let $f \colon \cFe \to \TT$ be any $\cS_1$-sequence. Take a large integer $M$, and consider the complete graph on $\NN$, coloured so that the edge $\{i,i'\}$ (where $i < i'$) is assigned colour $m \in [M]$ precisely when $f(\{i,i+1,\dots,i'-1\}) \in \left[ \frac{m}{M}, \frac{m+1}{M}\right) \bmod{1}$. By Ramsey's theorem, this graph has an infinite monochromatic clique $\{i_1,i_2,\dots\}$ with some colour $m$. Putting, as usual, $\a_j = \{i_j, \dots, i_{j+1}-1\}$, we obtain a $\cS_1$-subsequence $\tilde{f}(\b) = f(\a_\b)$ such that for each $\b \in \cF$, $\tilde{f}(\b) \in \left[ \frac{m}{M}, \frac{m+1}{M}\right) \bmod{1}$. In other words, the sequence $\tilde f$ is approximately constant. Note, however, that this phenomenon does not immediately generalise to $\cS_k$-sequences with $k \geq 2$.
\end{example}

\begin{example}\label{SEQ:ex:quadratic}
	Consider the degree $2$ polynomial $\cS_2$-sequence $f \colon \cFe \to \TT$ given by 
	$$f(\a) = \frac{1}{2} \abs{ \set{ \{i,j\} \subset \a}{ 0 \leq \abs{i-j} \leq 2 }} \bmod{1}.$$ 
	
	It is not difficult to check that if $\a$ is a (non-empty) arithmetic progressions with step $2$, then $f(\a) = \frac{1}{2}$. Moreover, if $\a, \a'$ are two disjoint progressions with step $2$, then $f(\a \cup \a') = f(\a) + f(\a') + \frac{1}{2} = \frac{1}{2}$ of $\a,\a'$ either overlap ($\min \a < \min \a' < \max \a < \max \a'$) or touch ($\max \a + 2 = \min \a'$); and $f(\a \cup \a') = f(\a) + f(\a') = 0$ if $\a$ and $\a'$ are distance $> 2$ apart. It follows that the only $\cS_2$-subsequence of $f$ is $f$ itself.
	
	We now show a few $\cS_1$-subsequences which can be extracted from $f$. If $\a = \{i,i+1,i+2\}$ is an interval of length $3$, then by direct computation $f(\a) = 0$. More generally, if $\a, \a'$ are two disjoint intervals of length $3$, then $f(\a \cup \a') = \frac{1}{2}$ if $\a, \a'$ are adjacent, and $= 0$ otherwise. Putting $\a_j = \{3j, 3j+1, 3j+2\}$, we thus obtain the $\cS_1$-subsequence $\tilde f(\b) = \frac{1}{2} \abs{ \set{ \{i,j\} \subset \a}{ \abs{i-j} = 1 }} \bmod{1}$; for $\b \in \cS_1$, this simplifies to $\tilde f(\b) = \frac{1}{2} (\abs{\b}-1) \bmod{1}$.
		
	A slightly more involved construction takes $\a_j = \{5j +1, 5j +3, 5j+4, 5j+5\}$. One can check by hand that $f(\a_j) = f(\a_{j} \cup \a_{j+1}) = 0$. As a consequence of these identities, the $\cS_1$-subsequence $\tilde{f}(\b) = f(\a_\b)$ is just the constant $0$ sequence: $\tilde f(\b) = 0$, $\b \in \cFe$.
\end{example}

Having discussed a number of examples of sequences and their subsequences, we turn to properties that are preserved under the operation of taking a subsequence. 

It is clear that if $g$ is a $\cS_l$-subsequence of a $\cS_k$-sequence $f$ then $g(\cS_{l}) \subset f(\cS_{k})$. Moreover, if $h$ is a $\cS_m$-subsequence of the $\cS_l$-sequence $g$, then $h$ is also a $\cS_m$-subsequence of $f$, hence the relation of being a subsequence is transitive.

Also, a $\cS_l$-subsequence of a polynomial sequence to a filtered nilpotent group is again polynomial, with respect to the same filtration, and there is an analogous statement for polynomial maps to nilmanifolds. Instances of this appear in Examples \ref{SEQ:ex:rational}, \ref{SEQ:ex:linear}, \ref{SEQ:ex:quadratic}; we record the proof below. Note that the ``linear'' sequences $\a \mapsto g^{n_\a}$ are preserved under taking subsequences (direct substitution).

\begin{proposition}\label{prop:D:SEQ:poly->poly}
	Let $G_{\bullet}$ be a filtration on nilmanifold $G$, and suppose that a $\cS_k$-sequence $f\colon \cFe \to G$ is a polynomial sequence as in Definition \ref{def:poly(A->G)}. For any $l \leq k$, if $\tilde{f}$ is a $\cS_l$-subsequence of $f$, then $\tilde{f}$ is a polynomial sequence. 
	
	Likewise, if $G_{\bullet}/\Gamma$ is a nilmanifold and $\bar f\colon \cFe \to G/\Gamma$ is a polynomial sequence as in Definition \ref{def:poly(A->G/Gamma)} then for any $l \leq k$, if $\tilde{f}$ is a $\cS_l$-subsequence of $f$, then $\tilde{f}$ is a polynomial sequence.
	
\end{proposition}
\begin{proof}
	Recall that $\tilde f(\b) = f(\a_\b)$ for some \ip\ ring $\FU(\a_i)$. Because of Proposition \ref{prop:poly-FCAE}, it suffices to check that $\tilde{f}$ maps parallelepipeds $\cFe^{[m]}$ to the corresponding Host-Kra space $\HK^{m}(G_\bullet)$. But this is clear because $\b \mapsto \a_\b$ maps $\cFe^{[m]}$ to $\cFe^{[m]}$. In fact, the last statement holds without any assumptions on $\a_i$ apart from them being disjoint.
	
	The second part of the statement can be seen either as an immediate consequence of the previous one, or proved using an analogous argument, with $\HK^{m}(G_\bullet/\Gamma)$ in place of $\HK^{m}(G_\bullet)$.
\end{proof}

\subsection{Asymptotic subsequences}
We are also interested in an asymptotic notion of a subsequence. From this point, we restrict attention to sequences taking values in compact metric spaces 
(we could work in the larger generality of compact topological spaces, but we do not need to).

\begin{definition}\label{def:Sk-subsequence}
	Let $f \colon \cFe \to X$ be a $\cS_k$-sequence taking values in a compact metric space $X$. Then, a sequence $\tilde f \colon \cFe \to X$ is said to be an \emph{asymptotic $\cS_l$-subsequence} of $f$ if it is a pointwise limit of $\cS_l$-subsequences of $f$. That is, $\tilde f$ is an asymptotic $\cS_l$-subsequence of $f$ precisely when there exist $\cS_l$-subsequences $\tilde f_n$ of $f$ such that for each $\b \in \cFe$, $\tilde f(\b) = \lim_{n \to \infty} \tilde f_n(\beta)$. 
\end{definition}

\begin{remark}\label{remark:Sk-subsequence-convergence}
	Note that in the final line of Definition \ref{def:Sk-subsequence} above, we require convergence for all $\b \in \cFe$, as opposed to only $\b \in \cS_l$; compare with Remark \ref{remark:Sk-not-enough}. However, if we have convergence for $\b \in \cS_l$, then by a diagonal argument we may always assume that $\tilde f_n(\b)$ converge also for $\b \in \cS_l$.
\end{remark}

\begin{example}
	We revisit previously mentioned examples.
	
	 The rational $\cS_1$-sequence in Example \ref{SEQ:ex:rational} has the constant $0$ sequence as an ordinary $\cS_1$-subsequence. 
	 
	 In contrast, the general linear $\cS_1$-sequence from Example \ref{SEQ:ex:linear} has the constant $0$ sequence as an asymptotic $\cS_1$-subsequence, but not an ordinary subsequence (unless the coefficients $a_\gamma$ are chosen in a specific way).
	
	Example $\ref{SEQ:ex:S1}$ shows that any $\cS_1$-sequence has an asymptotic subsequence which is constant of $\cS_1$. 
\end{example}

Remarks similar to the ones we made about subsequences apply also to asymptotic subsequences. If $\tilde f$ is an asymptotic $\cS_l$-subsequence of a $\cS_k$-sequence $f$ then $\tilde{f}(\cS_{l}) \subset \cl( f(\cS_{k}))$. The property of being a polynomial map is also preserved under taking asymptotic subsequences. 

\begin{proposition}\label{prop:D:SEQ:poly->poly-asympt}
	Let $G_{\bullet}/\Gamma$ be a nilmanifold, and suppose that a $\cS_k$-sequence $\bar f\colon \cFe \to G/\Gamma$ is a polynomial sequence as in Definition \ref{def:poly(A->G/Gamma)}. For any $l \leq k$, if $\tilde{\bar f}$ is a $\cS_l$-subsequence of $\bar f$, then $\tilde{\bar f}$ is a polynomial sequence.
\end{proposition}
\begin{proof}
	This follows immediately from Proposition \ref{prop:D:SEQ:poly->poly} combined with compactness of $\poly(\cFe \to G_\bullet/\Gamma)$ (Theorem \ref{thm:poly-is-compact}).
\end{proof}

However, we stress that ``linear'' sequences $\a \mapsto g^{n_\a}\Gamma$ are not preserved by the operation of taking asymptotic subsequences, as already seen from Example \ref{ex:quadratic-approximate-general}.

\begin{example}\label{ex:Heisenberg}
	Take $G = \begin{pmatrix} 1 & \RR & \RR \\ 0 & 1 & \RR \\ 0 & 0 & 1 \end{pmatrix}$ be $2$-step nilpotent group consisting of upper-triangular matrices, and let  $\Gamma = \begin{pmatrix} 1 & \ZZ & \ZZ \\ 0 & 1 & \ZZ \\ 0 & 0 & 1 \end{pmatrix}$ consist of the matrices in $G$ with integer entries. Then $G/\Gamma$ is a nilmanifold, known as the Heisenberg nilmanifold. Denote by $L$ the set of all sequences $\bar f \colon \cFe \to G/\Gamma$ of the form $\bar f(\a) = h^{m_\a} \Gamma$ ($h \in G,\ m_i \in \NN$).
		
	Pick $g = \begin{pmatrix} 1 & 2\theta & 0 \\ 0 & 1 & 1 \\ 0 & 0 & 1 \end{pmatrix}$, so that we may compute
	$$g^n \Gamma = \begin{pmatrix} 1 & 2 n \theta & n(n+1)\theta \\ 0 & 1 & n \\ 0 & 0 & 1\end{pmatrix} \Gamma = \begin{pmatrix} 1 & \{ 2 n \theta \} & \{ - n(n-1)\theta \} \\ 0 & 1 & 0 \\ 0 & 0 & 1\end{pmatrix} \Gamma.$$
	
	Assuming that $\theta$ is irrational, for any $\e > 0$ we may choose $n_1,n_2$ so that $\fpa{ n_1 \theta } < \e,\ \fpa{ n_1^2 \theta } < \e,\ \fpa{ n_2 \theta } < \e,\ \fpa{ n_2^2 \theta } < \e$ and $\fpa{ n_1 n_2 \theta - \psi } < \e$ for some freely chosen $\psi$ (this is an instance of Weyl's equidistribution theorem). Passing to a subsequence, we conclude that $\cl L$ (in the topology of pointwise convergence) contains a sequence $\bar f$ with $\bar f(\{1\}) = \bar f(\{2\}) = e \Gamma$ and $\bar f(\{1,2\}) = \begin{pmatrix} 1 & 0 & - 2 \psi \\ 0 & 1 & 0 \\ 0 & 0 & 1 \end{pmatrix} \Gamma$. Suppose for the sake of contradiction that $\bar f(\a) = h^{m_\a} \Gamma$. Then, $h^{m_1}, h^{m_2} \in \Gamma$, so $h^{m_1 + m_2} \in \Gamma$ and $\bar f(\{1,2\}) = e \Gamma$ --- contradiction.
\end{example}

\begin{observation}
	Among the $\cS_k$-sequences taking values in a given compact metric space $X$, the relation of being an asymptotic $\cS_k$-subsequence is transitive and reflexive. Hence, it is a partial weak order.
	
	More generally, if $f$ is a $\cS_{k}$-sequence taking values in a compact metric space $X$, $g$ is an asymptotic $\cS_l$-subsequence of $f$, and $h$ is an asymptotic $\cS_{m}$-subsequence of $g$, then $h$ is an asymptotic $\cS_m$-subsequence of $f$.
\end{observation}
\begin{proof}
	Reflexivity is direct from the definition. It suffices to take $\a_i = \{i\}$. 
	
	For transitivity, we begin with noting that it will suffice to prove that for $f,g,h$ as above, if $h$ is a (non-asymptotic) $\cS_m$-subsequence of $g$, then $h$ is also an asymptotic $\cS_m$-subsequence of $f$ (this is because the set of asymptotic $\cS_m$-subsequences is the same as the closure of the set of $\cS_m$-subsequences in $X^{\cFe}$). 
	
	Suppose that $g(\b) = \lim_{n \to \infty} f_n(\b)$ where $f_n(\b) = f(\a^{(n)}_\b)$, and $h(\c) = g(\b_\c)$. Then clearly $h(\c) = \lim_{n \to \infty} f'_n(\gamma)$, where $f'_n(\c) = f(\a^{(n)}_{\b_\c})$ is a $\cS_m$-subsequence of $f$. Hence, $h$ is an asymptotic $\cS_m$-subsequence of $f$.
\end{proof}

\begin{remark}\label{rem:subsequence-relation-not-antisymmetric}
The relation of being $\cS_k$-subsequence (let alone asymptotic $\cS_k$-subsequence) is not, however, anti-symmetric. This is shown by a simple example of a pair of $\ZZ/{k}\ZZ$-valued $\cS_k$-sequences $f(\a) = \min \a \pmod{k},\ g(\a) = \min \a+1 \pmod{k}$, which are easily seen to be $\cS_k$-subsequences of one another (put $\a_i := \{i+1\}$ and $\a'_i = \{i+k-1\}$). More generally, follows directly from Definition \ref{def:Sk-subsequence} that all $\cS_k$-subsequences of $f$ take the form $\a \mapsto \sigma(\min \a \pmod k)$, where $\sigma$ is a permutation of $\ZZ/k\ZZ$ (identified with $[k]$), and any sequence of this form is a $\cS_k$-subsequence of $f$ (take $\a_{a k + l} = \{2a k + \sigma(l), (2a+1) k + \sigma(l) \}$). Hence, it can be checked that $f$ is a $\cS_k$-subsequence of any of its $\cS_k$-subsequences.
\end{remark}

\subsection{Stable sequences} \label{sec:SEQ:Stability}
We will often find ourselves in the position of working with a $\cS_k$-sequence taking values in a compact metric space, where we may freely restrict to asymptotic $\cS_k$-subsequences. Hence, it is of interest to enquire into the possible simplest objects that can be obtained through such restrictions. This motivates the following definition.

\begin{definition}\label{def:D:SEQ:stable}
	Let $f$ be a $\cS_k$-sequence taking values in a compact metric space $X$. Then $f$ is said to be \emph{stable} if for any asymptotic $\cS_k$-subsequence $g$ of $f$, $f$ is again an asymptotic subsequence of $g$.  
\end{definition}

\begin{example}\label{ex:stable}
	Any constant sequence is automatically stable. 
	
	Example \ref{SEQ:ex:linear} shows that any linear $\cS_1$-sequence $f \colon \cFe \to \TT$ (with $f(\emptyset) = 0$) has the constant $0$ sequence as a $\cS_1$-subsequence. Hence, the only stable linear $\cS_1$-sequence is the constant $0$ sequence.
	
	On the other hand, Example \ref{SEQ:ex:quadratic} shows that there are non-trivial quadratic $\cS_2$-sequences $f \colon \cFe \to \TT$ which are stable. To extract a simpler subsequence from $f$, one has to pass to a $\cS_1$-subsequence.
\end{example}

\begin{lemma}\label{lem:stable-Sk-seq-exist}
	Let $f \colon \cFe \to X$ be a $\cS_k$-sequence taking values in a compact metric space $X$. Then, there exists an asymptotic $\cS_k$-subsequence $g$ of $f$ which is stable in the sense of Definition \ref{def:D:SEQ:stable}. 
\end{lemma}
\begin{proof}
	This is standard application of the Kuratowski-Zorn Lemma. Consider the set of all asymptotic $\cS_k$-subsequences of $f$, weakly partially ordered by the relation of being a $\cS_k$-subsequence. Clearly, any minimal element for this ordering will be a stable $\cS_k$-sequence, so it will remain to check that any chain has a 	lower bound. If $\cC$ is a chain, then by a standard diagonalising argument, we may assume that for each $\a \in \cFe$, the sequence $\{h(\a)\}_{h \in \cC}$ converges to some limit $g(\a)$. It is now routine to check that $g$ is an asymptotic $\cS_k$-subsequence of $f$.
\end{proof}

To formulate (and prove) a structural result for stable $\cS_k$-sequences, it will be useful to introduce a way of ``shuffling'' $\cS_k$-sequences. Let $\pi$ be a permutation of $[k]$, and let us extend $\pi$ to a permutation of $\NN$ by $\pi(ak + l) = \pi(l) + ak$ ($a \in \NN_0, l \in [k]$). Now, put $\a_i := \{\pi(i)\}$ and define the $\cS_k$-sequence $\tilde f_\pi$ by $\tilde f_\pi(\b) = f(\a_\b)$. In plainer terms, $\tilde f_\pi(\b)$ is obtained from $f$ by permuting $\NN$ according to $\pi$. We will call such sequence a \emph{shuffle} of $f$. Note that a shuffle of a shuffle is again a shuffle, and the original sequence is a shuffle of any of its shuffles.

\begin{remark}\label{remark:shuffles-and-subsequences}
Note that in general, the shuffle $\tilde f_\pi$ need not be a $\cS_k$-subsequence of $f$, since $\b \mapsto \a_\b$ usually does not map $\cS_k$ to $\cS_k$. However, it will later be of importance that if $g(\c)$ is a $\cS_k$-subsequence of $\tilde f_\pi$, taking the form $g(\c) = \tilde f_\pi(\b_\c)$ where $\abs{\b_j} \geq 2$ for all $j$, then $g$ \emph{is} a $\cS_k$-subsequence of $f$. Indeed, it will suffice to check that for $j,j'$ with $j < j' \leq j+k$, the set $\a_{\b_{\{j,j'\}}} = \set{ \pi(i)}{ i \in \b_j \cup \b_{j'}}$ has gaps $\leq k$. By definition, $\b_j = \{i_j, \dots, i_{j+k} - k\}$ and $\b_{j'} = \{i_{j'}, \dots, i_{j'+k} - k\}$. Since $\pi$ maps pairs of points at distance exactly $k$ apart into pairs of points exactly $k$ apart, in the case $j' = j+k$ we are done. If $j' < j+k$, then there is some $i \in \b_j$ such that $i_{j'} < i < i_{j'} + k$, so $\pi(i')$ is at distance $\leq k$ from one of $\pi(i_{j'}), \pi(i_{j'}+k)$. 
\end{remark}
  
In the sequel, we only use a special case of the following proposition, which deals with a more specific type of a subsequence.

\begin{proposition}[Structure theorem for stable sequences]
\label{prop:structure-theorem-stable}
	Let $f$ be a stable $\cS_k$-sequence taking values in a compact metric space $X$. Suppose that $g$ is an asymptotic $\cS_k$-subsequence of $f$. Then $g$ is a shuffle of $f$.
	
	In the special case when $g$	takes the form $g(\b) = f(\a_\b)$, where $\a_j = \{i_j, i_j+{k}, \dots, i_{j+{k}} -{k} \}$ for an increasing sequence $(i_j)_{j = 1}^\infty$ with $i_{j} \equiv j \pmod{k}$, we have a stronger conclusion $g = f$.
	
\end{proposition}

We will obtain the above result as a consequence of the following more precise statement. It will be convenient to use the notion of ultrafilters and the corresponding limits; for an accessible introduction we refer the reader for instance to \cite{Bergelson2010d} or the early chapters of \cite{HindmanStrauss-book}.

An ultrafilter on $\NN$ is a collection $p \subset \cP(\NN)$ which is closed under finite intersections and taking supersets, not containing $\emptyset$, and maximal with respect to aforementioned properties.  Alternatively, one may identify an ultrafilter $p$ with a $\{0,1\}$-valued, finitely additive measure $\mu$ on $\cP(\NN)$, so that $p$ is the family of sets with measure $1$. If $p$ is an ultrafilter and $A \subset \NN$ then exactly one of $A$ and $\NN \setminus A$ belongs to $p$.

Trivial examples of ultrafilters include the \emph{principal} ultrafilters, which take the form $p_x = \set{A \in \cP(S)}{x \in S}$. Ultrafilters not of this form are called non-principal; only non-principal ultrafilters will be of use to us. Non-principal ultrafilters are known to exists, but since their construction relies on the axiom of choice in an essential way, no explicit examples can be given. 

If $f\colon \NN \to X$ is a function taking values in a compact space $X$, then there exists a unique point of $X$, denoted $\plim_{n} f(n)$, such that for each of its open neighbourhoods $U$, the preimage $\set{n \in \NN}{f(n) \in U}$ belongs to $p$. (A more general definition is possible, but not needed in our applications.)

The limits along ultrafilters share many of the familiar properties of the usual limits. For instance, if $\phi \colon X \to Y$ is a continuous map between compact spaces and $f\colon \NN \to X$, then $\plim{p}{n} \phi \circ f(n) = \phi\bra{ \plim{p}{n} f(n) }$. In particular, if $f,g \colon \NN \to X$ take values in a compact topological group, then $\plim{p}{n} f(n) \cdot g(n) = \bra{ \plim{p}{n} f(n) } \cdot \bra{ \plim{p}{n} g(n) }$.

We are now ready to use the ultrafilters to extract particularly simple asymptotic $\cS_k$-subsequences. 

\begin{lemma}[Ultrafilter restriction lemma]
\label{lem:ultrafilter-restriction-lemma}

Let $k$ be an integer, and let $p = (p_1,p_2,\dots,p_k)$ be a $k$-tuple of non-principal ultrafilters with $n \equiv j \pmod{k}$ for $p_j$-almost all $n$. For a sequence $ i = (i_j)_{j=1}^\infty$ with $i_j \equiv j \pmod{k}$, consider the intervals $\a_j(i)$ given by $$\a_j(i) = \{i_j , i_j+k, \dots, i_{j+k} -k \},$$ 
and put as usual $\a_{\b}(i) := \bigcup_{j \in \b} \a_j(i)$.

For a $\cS_k$-sequence $f$ taking values in a compact metric space $X$, define the $\cS_k$-sequence $\tilde f_p$ as
\begin{equation}
	\tilde f_p(\b) := \plim{p_1}{i_1}\ \plim{p_2}{i_2}\ \plim{p_3}{i_3} \dots \plim{p_1}{i_{k+1}} \plim{p_2}{i_{k+2}} \dots  f(\a_\b(i)).
	\label{eq:ultrafilter-restriction}
\end{equation}
(Above, the limit of $i_m$ is taken along $p_{m \bmod k}$. The limits are taken over all relevant indices $i_j$, i.e.\ those with $j = j' + l$, $j \in \b,\ 0 \leq l \leq k$; there are finitely many of those.)

Then $\tilde{f}_p $ is an asymptotic $\cS_k$-subsequence of $f$.
\end{lemma}
\begin{proof}
	It will suffice to construct, for any sequence of neighbourhoods $U_\b$ of $\tilde f_p(\b)$ (with $U_\b = X$ for all but finitely many $\b$), an increasing sequence $i = (i_j)_{j=1}^\infty$ such that $f(\a_\b(i)) \in U_\b$ for each $\b$. For convenience, let $p_m := p_{m \bmod k}$ for $m > k$.

	We construct $i_j$ inductively. Unwinding the definitions of limits, for each $\b$, there exists a sequence of (families of) sets $A_\b^{(0)}$, $A_\b^{(1)}(i_1)$, $A_\b^{(2)}(i_1,i_{2}), \dots, A_\b^{(r)}(i_1,i_{2}, \dots, i_{r})$, (with $r$ dependent on $\b$) such that, firstly, for each $l$ we have $A_\b^{(l)}(i_1,i_{2},\dots,i_{l}) \in p_{l+1}$ and, secondly, if $i_{l+1} \in A_\b(i_1,i_{2},\dots,i_{l})$ for each $0 \leq l \leq r$, then $f(\a_\b(i)) \in U_\b$. 
	
	By assumption, $\set{i \in \NN}{i \equiv j \pmod{k}} \in p_j$, so without loss of generality we may suppose that for any $l$, and any $i \in A_\b^{(l)}(i_1,i_{2},\dots,i_{l})$ we have $i \equiv l+1 \pmod{p}$ and $i > i_l$. For simplicity of notation, let us put $A_\b^{(l)}(i_1,i_{2}, \dots, i_{l}) = \set{i \in \NN}{i \equiv l+1 \pmod{p}, \ i > i_l}$ for all $l > r$, so that $A_\b^{(l)}$ are defined for all $r$. 
	
	It now becomes clear how the sequence $i_j$ needs to be constructed. Let $$A^{(l)}(i_1,\dots,i_l) := \bigcap_{\b : U_\b \neq X} A^{(l)}_\b(i_1,\dots,i_l),$$ where the intersection is taken over $\b \in \cF$ with $U_\b \neq X$. Begin with taking any $i_1 \in A^{(0)}$. This can be done because $\emptyset \neq A^{(0)} \in p_1$. In general, for each $l$, take arbitrary $i_{l+1} \in A^{(l)}(i_1,i_{2},\dots,i_{l})$, which can always be done because the sets $A^{(l)}(i_1,i_{2},\dots,i_{l})$ are guaranteed to be in $p_l$, and in particular be non-empty.
\end{proof}

\begin{proof}[Lemma \ref{lem:ultrafilter-restriction-lemma} implies Proposition \ref{prop:structure-theorem-stable}.]

Pick a $k$-tuple of ultrafilters $p$ as in Lemma \ref{lem:ultrafilter-restriction-lemma}. By  Lemma \ref{lem:ultrafilter-restriction-lemma}, $\tilde f_p$ is an asymptotic $\cS_k$-subsequence of $f$, and by stability $f$ is an asymptotic $\cS_k$-subsequence of $\tilde f_p$.

Let $g$ be a (non-asymptotic) $\cS_k$-subsequence of $\tilde f_p$. We claim  that $g$ is a shuffle of $\tilde f_p$. Once this is shown, it will follow that also all asymptotic $\cS_k$-subsequences of $\tilde f_p$ are shuffles of $\tilde f_p$ (this is because there are finitely many shuffles of $\tilde{f}_p$). In particular, $f$ is a shuffle of $\tilde f_p$, and consequently all (asymptotic) $\cS_k$-subsequences of $f$ are shuffles of $f$, as needed. Hence, it remains to prove the claim.

By definition, $g$ takes the form $g(\c) = \tilde f_p (\b_\c)$. Here, the sequence $\b_m$ takes  the form $\b_m =  \{j_m,j_m+k,\dots,j_{m+k}-k\}$ where $(j_m)_{m = 1}^\infty$ is a sequence of integers and $j_m \bmod k$ depends only on $m \bmod k$. Replacing $g$ with a shuffle, we may assume that $j_m \equiv m \pmod{k}$.

For a sequence $i = (i_k)_{k=1}^\infty$, let $\a_\b(i)$ be as in the Lemma \ref{lem:ultrafilter-restriction-lemma}. We may now observe that
$$\a_{\b_m}(i) = \bigcup_{j \in \b_m} \{i_{j},\dots, i_{j+k}-k\} = \{i_{j_m},\dots,i_{j_{m+k}}-k\}
.$$
Let $i'_m = i_{j_m}$ and $\a'_{m} = \{i_m',\dots,i_{m+k}'-k\}$, and let $\a_\c'$ be defined accordingly. 
The key point is that if we disregard inconsequential indices $i_j$ with $j \neq j_m$, then the limit defining $g(\c) = \tilde f_p (\b_\c)$ becomes identical with the limit defining $\tilde f_p (\c)$. More precisely, we may write:
\begin{align*}
	g(\c) = \tilde f_p (\b_\c) 
	&= \plim{p_1}{i_1}\ \plim{p_2}{i_2} \plim{p_3}{i_3} \dots f(\a_{\b_\c}(i)) 
	\\ &= \plim{p_1}{i_1}\ \plim{p_2}{i_2} \plim{p_3}{i_3} \dots f(\a_\c'(i)) 
	\\ &= \plim{p_1}{i_1'}\ \plim{p_2}{i_2'} \plim{p_3}{i_3'} \dots f(\a_\c'(i))
	\\ &= \plim{p_1}{i_1}\ \plim{p_2}{i_2} \plim{p_3}{i_3} \dots f(\a_\c(i)) 
	= \tilde f_p (\c). 
\end{align*}

Note that if $j_m \equiv m\mod{k}$ for all $m$ to begin with, then the application of the shuffle inside the proof is unnecessary; we immediately find $g = \tilde f_p$. Applying the shuffle which takes $\tilde f_p$ to $f$, we conclude that the analogous statement holds $f$, i.e.\ all $\cS_k$-subsequences of $f$ of the special form as above are again $f$. In particular, by inspection of formula \eqref{eq:ultrafilter-restriction}, we conclude that $\tilde{f}_p = f$.
\end{proof}

\begin{corollary}\label{cor:dilation-invariance}
	Let $A$ be a \sg{k} set, and pick any $m \in \NN$. Then $A \cap m \NN$ is a \sg{k} set. Likewise, if $B$ is a \sgd{k} set, then $B \cap m \NN$ is a \sgd{k} set.
\end{corollary}
Note that similar facts for \ip\ sets are well known.
\begin{proof}
	By the definition of $A$ being  a \sg{k} set, there is a $\cS_k$-sequence $n_\a$ such that $n_\a \in A$ for $\a \in \cS_k$. Passing to a $\cS_k$-subsequence, we may assume without loss of generality that the sequence $(n_\a \bmod{m})$ is stable (note that because the target space is finite, there is no need to use asymptotic subsequences). Thus, for any $i$ we have
	$$n_{i} \equiv n_{\{i,i+k\}} \equiv n_{i} + n_{i+k} \equiv 2n_i \pmod{m},$$
	whence $n_i \equiv 0 \pmod{m}$, and more generally $n_\a \equiv 0 \pmod{m}$ for all $\a \in \cS_k$.
	
	For the second statement, it suffices to check that $B \cap m \NN \cap A \neq \emptyset$ whenever $A$ is \sg{k}. But this is clear, since $m \NN \cap A$ is \sg{k} and $B$ is \sgd{k}.
\end{proof}

\subsection{Stable polynomials}\label{sec:SEQ:polynomials}

Among all $\cFe$-sequences, we will be particularly interested in polynomial maps into the torus. Recall that an (asymptotic) subsequence of a polynomial map is again polynomial because of Proposition \ref{prop:poly-FCAE}. Hence, we may in most cases assume that the polynomial we are working with is stable by passing to a subsequence. In this situation, we can obtain the following refinement of Proposition \ref{prop:structure-of-poly(F->T)}. For a set $\gamma \in \cFe$, we define the diameter $\diam (\gamma) = \max \gamma - \min \gamma$. By a slight abuse of notation, we write $\gamma + k$ for the sumset $\set{ i + k }{ i \in \gamma}$.

\begin{proposition}[Structure of stable polynomials $\cFe \to \TT^m$]
\label{cor:structure-of-stable-poly(F->T)}
	Suppose that ${f} \colon \cFe \to \TT^m = \RR^m/\ZZ^m$ is a polynomial of degree $d \leq k$, which is a stable $\cS_k$-sequence (in the sense of Definition \ref{def:D:SEQ:stable}). Then ${f}$ admits a representation of the form:
	\begin{equation}	
		{f}(\a) = \sum_{\substack{ \c \subset \a}} a_{\c},
		\label{DEF:eq:form-of-stable-poly(F->T)}
	\end{equation}
	where $\a_{\c} \in \TT^m$ are constants, which further satisfy $a_\gamma = 0$ if $\diam (\gamma) > k$ or $\abs{\gamma} > d$,
 and are periodic in the sense that $a_{\gamma+k} = a_{\gamma}$. 
\end{proposition}

\begin{proof}
	It is already shown in Proposition \ref{prop:structure-of-poly(F->T)} that ${f}$ admits a representation as in \eqref{DEF:eq:form-of-stable-poly(F->T)} with $a_\gamma = 0$ if $\abs{\gamma} > d$. 
	
	The periodicity condition $\nobar f(\a) = \nobar f(\a + k)$ follows immediately from applying Proposition \ref{prop:structure-theorem-stable} to the $\cS_k$-subsequence of $f$ given by the shift: $\nobar{g}(\a) = \nobar f(\a+k)$. (Alternatively, it is also a consequence of the form of the limit in \eqref{eq:ultrafilter-restriction} and the fact that $\nobar{f} = \tilde{\nobar{f}}_p$). Since the coefficients $a_\gamma$ are uniquely determined by $\nobar f$, we also have $a_{\gamma} = a_{\gamma + k}$.
	
	For the vanishing of coefficients $a_{\gamma}$ with $\diam(\gamma) > k$, we proceed by induction on $d$, the case $d = 1$ being trivial. Note that if the claim holds for some $d$, then for any degree $d$ stable polynomial $f$ and for any $\a,\b$ with $\max \a < \min \b - k$ we have $f(\a \cup \b) = f(\a) + f(\b) - f(\emptyset)$.
	
	To prove the claim for arbitrary $d \leq k$, fix some $j_1 < j_2 - k$. By an inclusion-exclusion type  argument it will suffice to prove that for any $\a \in \cFe$ with $\a \subset (j_1,j_2)$ we have
	\begin{align}
	\sum_{\substack{ \gamma \subset \a \cup \{j_1,j_2\} \\ j_1,j_2 \in \gamma }} a_\gamma = 0.
	\label{SEQ:eq:001}
	\end{align}
	Indeed, for $\delta \subset [j_1,j_2]$ with $j_1,j_2 \in \delta$ we have the identity 
	$$a_{\delta} = \sum_{\a \subset \delta \setminus \{j_1,j_2\} } (-1)^{\abs{\delta \setminus \a}} \sum_{\substack{ \gamma \subset \a \cup \{j_1,j_2\} \\ j_1,j_2 \in \gamma }} a_\gamma,$$ so \eqref{SEQ:eq:001} implies that $a_\delta = 0$.
	
	We first consider the case $\a = \emptyset$. Let $j_3 = j_2 + k$. Since $\Delta_{j_3} f$ is a polynomial of degree $d-1$ we may write:
	$$
	\Delta_{j_3} f(\{j_1,j_2\}) = 	\Delta_{j_3} f(\{j_1\}) + \Delta_{j_3} f(\{j_2\}) - \Delta_{j_3} f(\emptyset)
	$$
	which can be rewritten in simpler terms as:
	\begin{align*}
	f(\{j_1,j_2,j_3\}) - f(\{j_1,j_2\}) 
	&= f(\{j_1,j_3\}) - f(\{j_1\}) 
	\\&+ f(\{j_2,j_3\}) - f(\{j_2\}) - f(\{j_3\}) + f(\emptyset).
	\end{align*}
	Using stability of $f$ combined with Proposition \ref{prop:structure-theorem-stable}, we may replace each occurrence of $\{j_2,j_3\}$ or $\{j_3\}$ with $\{j_2\}$. The above equation now simplifies to 
	$$f(\{j_1,j_2\}) = f(\{j_1\}) + f(\{j_2\})-f(\emptyset).$$
Writing out $f$ in coordinates from \eqref{DEF:eq:form-of-stable-poly(F->T)} and cancelling repeating terms, this gives the sought formula \eqref{SEQ:eq:001}.

	We now consider arbitrary $\a \subset (j_1,j_2)$. The map $\Delta_\a f$ is a polynomial of degree $d-1$ so by the inductive hypothesis 
	$$\Delta_\a f(\{j_1,j_2\}) = \Delta_\a f(\{j_1\}) + \Delta_\a f(\{j_2\}) - \Delta_\a f(\emptyset).$$
	In particular, we have
	\begin{align*}
		f(\a \cup \{j_1,j_2\}) 
		&= f(\a \cup \{j_1\}) + f(\a \cup \{j_2\}) 
		\\ &+ f(\{j_1,j_2\}) - f(\{j_1\}) - f(\{j_2\}) - f(\a)+ f(\emptyset).
	\end{align*}
	Using the previous step, we may simplify this to 
	$$
		f(\a \cup \{j_1,j_2\}) = f(\a \cup \{j_1\}) + f(\a \cup \{j_2\})  - f(\a),
	$$
	which is equivalent to \eqref{SEQ:eq:001}. 
\end{proof}

\section{Basic results}\label{sec:SPL}

\subsection{Abelian case}\label{sec:SPL:d=1}
We observe that with the theory developed so far, the case $d = 1$ of our main result becomes trivial. This case is already well known (compare Example \ref{SEQ:ex:linear}), but we discuss it here as a source of motivation.

Recall that any $1$-step nilmanifold is in fact a compact abelian Lie group, and hence a product of a torus $\TT^m$ and a finite abelian group. In all cases, it can be viewed as a submanifold of a (possibly higher-dimensional) torus, so we may restrict our attention to tori (see also Section \ref{sec:DEF:Connected}).

\begin{proof}[Proof of Theorem \ref{thm:B}, case $d = 1$]
	We need to check that for any linear (i.e.\ polynomial degree $1$) $\cS_1$-sequence, ${f} \colon \cFe \to \TT^m$ with $f(\emptyset) = 0$, there are points $ f(\a)$ with $\a \in \cS_1 $ arbitrarily close to $0$. In fact, we prove somewhat more, namely that the constant sequence $0$ is an asymptotic $\cS_1$-subsequence of $f$, and hence such sets $\a$ are in rich supply.

	Let $ g$ be a stable asymptotic $\cS_1$-subsequence of $ f$, as introduced in Section \ref{sec:SEQ:Stability}. By Proposition \ref{cor:structure-of-stable-poly(F->T)}, we may write $ g$ in the form:
	$$
		g(\a) = \sum_{i \in \a} a_{i} = \abs{\a} t,
	$$
	where $t = a_{1}$. Stability of $ g$ further implies:
	$$
		t = g(\{1\}) = g( \{1,2\}) = 2 t,
	$$	
	and hence $t = 0$. If follows that $g(\a) = 0$ for each $\a$, as needed. 
\end{proof}

\subsection{Case $d=2$}\label{sec:SPL:d=2}
\setcounter{step}{0}

We now move on to quadratic polynomials, which already shows some of our main ideas.

\begin{proof}[Proof of Theorem \ref{thm:B}, case $d = 2$]

It will suffice to show that whenever $\bar{f} \in \poly(\cFe \to G_\bullet/\Gamma)$ is a $\cS_2$-sequence  with $\bar f(\emptyset) = e \Gamma$, where $G_\bullet/\Gamma$ is a length $2$ nilmanifold, there are $\a \in \cS_2$ such that $\bar f(\a)$ is arbitrarily close to $e \Gamma$. Hence, we prove that any \nilbohrz{2} set is \sgd{2}.

\begin{step}[Model problem]\label{step:1@thm:case-d-2}
	Let $f\colon \cFe \to \TT^m$ be a polynomial of degree $2$ with $f(\emptyset) = 0$. Then, for any $\e > 0$, there exists $\a \in \cS_2 $ such that $\norm{f(\a)} < \e$.
\end{step}
\begin{proof}
	The form of the claim is such that we can freely restrict $f$ to an asymptotic $\cS_2$-subsequence. Hence, we may without loss of generality assume that $f$ is stable in the sense of Definition \ref{def:D:SEQ:stable}. Now, by Proposition \ref{cor:structure-of-stable-poly(F->T)}, $f$ takes the form:
$$	
	f(\a) = \sum_{\substack{ \c \subset \a }} a_{\c},
$$
	where $\a_{\c} \in \TT$ are constants, which further satisfy $a_\gamma = 0$ if $\diam(\gamma) > 2$ or $\abs{\gamma} > 2$, 
	and $a_\gamma$ are periodic in the sense that $a_{\gamma+2} = a_{\gamma}$. 
	
	Taking into account above properties of $a_\gamma$, there are only $6$ meaningful coefficients, namely $a_1, a_2, a_{12}, a_{23}, a_{13}, a_{24}$. (We omit the curly brackets to avoid obfuscating notation; hence e.g. $a_{13}\equiv a_{\{1,3\}}$.)

	We record some of the relations among the $\a_{\gamma}$, which easily follow from stability:
\begin{align}
	a_{13} &= -a_1, & a_{24} &= - a_2, &
	a_{12} + a_{23} &= 0. 
	\label{SMPL:eq:001}
\end{align}
	More exactly, these follow from the identities $f(\{1,3\}) = f(\{1\})$,  $f(\{2,4\}) = f(\{2\})$ and  $f(\{1,2,3,4\}) = f(\{1,2\})$ respectively.

	Our general strategy at this point is to start with a sufficiently \emph{generic} set $\a \in \cS_2$, and check that it can be perturbed to a set $\a'$ such that $f(\a') \simeq 0$. A set $\a$ will be highly generic if it contains any possible \emph{pattern} a large number of times.
	
	To make these ideas precise, let us say that a pattern of length $M$ is a set $\pi \in \cS_2$ such that $\pi \subset [M] = \{1,2,\dots,M\}$. We will say that $\pi$ \emph{appears in $\a$ at position $n$} if $n \equiv 0 \pmod{2}$ and $(\a - n) \cap [M] = \pi$. (Note that this condition depends not only on $\a$ and $\pi$ but also on $M$.) Every pattern will appear in \emph{some} $\a \in \cS_2$ (e.g. $\a = \pi$), but there are some patterns which appear only boundedly many times. For instance, the length $3$ pattern $\pi = \{1\}$ can appear only once in any $\a \in \cS_2$, since if $\pi$ appears at position $n$, then $n+1 = \max \a$. Thus, we say that $\a$ is $(M,N)$-generic if for any pattern $\pi$ of length $M$ one of the following is true: either $\pi$ appears at most $C = C(\pi)$ times in \emph{any} $\b \in \cS_2$, or $\pi$ appears at least $N$ times in $\a$.
	
	It is a simple but useful fact that for any $M,N$ there exist $\a \in \cS_2$ which are $(M,N)$-generic. We will discuss a related issue in more detail in Section \ref{sec:MAIN}. Note also that if $\a,\a' \in \cS_2$ are close in the sense that $\abs{\a \triangle \a'} \leq L$ for some integer $L$, then if $\a$ is $(M,N)$-generic, then also $\a'$ is $(M,N')$-generic, where $N' = N - C(L,M)$ for some constant $C$ independent of $N$. We are interested in the regime where $N \to \infty$, but $M$ takes a fairly small value. In fact, it will suffice to take $M = 5$.
	
The class of generic sets is slightly too large for our purposes, for somewhat mundane reasons. We will call a set $\a$ \emph{well-formed} if additionally $\max \a \equiv \min \a \pmod{2}$, and work mostly with well-formed $\a$. For instance, $\a = \{1,2,3,5\}$ is well-formed and $f(\a) = 3a_1 + a_2 + 2a_{13} + a_{12} + a_{23} = a_1 + a_2$, while $\b = \{1,2,3,4\}$ is not well-formed and $f(\b) = 2a_1 + 2a_2 + a_{13} + a_{24} +  2a_{12} + a_{23} = a_1+a_2+a_{12}$; the appearance of the coefficient $a_{12}$ would cause problems.
	
	Informally, we would like to consider the set of all possible values of $f(\a)$ for all highly generic and well-formed $\a$. Hence, we define $\Sigma$ to be the set of those $s \in \TT$ such that for any $\e > 0$ and any $N$ there exists $(5,N)$-generic, well-formed $\a$ such that $\norm{f(\a) - s} < \e$. Note that $\Sigma$ is closed. Secondly, we want to take $\Delta$ to be the set of the perturbations of $f(\a)$ which are always feasible if $\a$ is sufficiently generic. Precisely, we declare $t \in \Delta_0$ if and only if there exists constants $L$ and $N_0$ such that for all $N \geq N_0$ and all $N$-generic, well-formed $\a$, we can find well-formed $\a'$ with $\abs{\a \triangle \a'} \leq L$ and $f(\a') = f(\a) + t$. Finally, put $\Delta := \cl(\Delta_0)$.	We note some simple properties of $\Sigma$ and $\Delta$.
	
	 Firstly, we have $\Sigma + \Delta = \Sigma$ and $\Delta$ is a group. Because $0 \in \Delta$ and $\Sigma$ is closed, to prove the first equality it will suffice to prove the inclusion $\Sigma + \Delta_0 \subset \Sigma$. Take $s \in \Sigma$ and $t \in \Delta_0$. For any $N,\e$ we may select $(5,N)$-generic, well-formed $\a$ such that $\norm{f(\a) - s} < \e$. We may (assuming that $N$ is large enough) find $\a'$ with $\abs{\a \triangle \a'} < L$ for a constant $L$ dependent only on $t$ such that $f(\a') = f(\a) + t$. Hence, $\norm{f(\a') - (s+t)} < \e$, and $\a'$ is $N'$ generic, where $N' = N - C(L) \to \infty$ as $N \to \infty$. This finishes the proof of the first equality. A proof that $\Delta + \Delta = \Delta$ follows along similar lines. If $t,t' \in \Delta_0$ then for sufficiently generic $\a$, there are $\a',\a''$ such that $\abs{\a \triangle \a''} \leq \abs{\a \triangle \a'} + \abs{\a' \triangle \a''} \leq L(t) + L(t')$ and $f(\a'') = f(\a') + t' = f(\a) + t+t'$. To prove that $\Delta$ is a group, it remains to check that $\Delta = -\Delta$. Take $t \in \Delta$. There exists a sequence $n_j$ such that $n_j t \to 0$ as $j \to \infty$. Hence $(n_j - 1) t \to -t \in \Delta$, as needed.
	
	Secondly, we exhibit generating sets for $\Sigma$ and $\Delta$. For any well-formed $\a$, we claim that $f(\a) \in  \ZZ a_1 + \ZZ a_2$. Indeed, the analogous statement involving $a_\gamma$ for all $\gamma \in \{1,2,12,23,13,24\}$ would obviously be true. We may eliminate $a_{13}$ and $a_{24}$ using \eqref{SMPL:eq:001}. Next, we may notice that the expression for $f(\a)$ contains the same number of occurrences of $a_{12}$ and $a_{23}$, so these can also be eliminated using \eqref{SMPL:eq:001} ($f(\a) = a_1 + a_2$ when $\a$ is an interval, and removing an element from $\a$ preserves the difference between the number of times $a_{12}$ and $a_{23}$ appear in $f(\a)$; this is the only point where we use the fact that $\a$ is well-formed). Hence, $\Sigma \subset \cl( \ZZ a_1 + \ZZ a_2 )$.
	
	We next consider $\Delta$. If $\a$ is $(5,N)$-generic and $N$ is sufficiently large, then we may find an occurrence of the length $5$ pattern $\{2,4\}$. We may form $\a'$ by replacing this pattern with $\{2,3,4\}$ (i.e.\ put $\a' := \a \cup \{n+3\}$, where $n$ is the position where the original pattern appears in $\a$). An easy calculation shows that $f(\a') - f(\a) = a_{12} + a_{23} + a_{1} = a_{1}$. Hence $a_{1} \in \Delta$. By a symmetric argument, $a_{2} \in \Delta$. Thus, $\Delta \supset  \cl( \ZZ a_1 + \ZZ a_2 )$. 
	
	Combining the two inclusions, we conclude that $\Sigma = \Delta = \cl( \ZZ a_1 + \ZZ a_2 )$. This finishes the argument: Since $0 \in \Sigma$, we may find some $\a \in \cS_2$ (which incidentally is highly-generic and well-formed) such that $\norm{f(\a)} < \e$, where $\e$ is as small as we please.	
\end{proof}

\begin{remark}
	Our proof can be used to find $\cS_0$-subsequence which is identically $0$, rather than a single set. Note, however, that we are unable to get a $\cS_1$-subsequence.
\end{remark}

\begin{remark}
	A more condensed proof of Step \ref{step:1@thm:case-d-2} is possible. Let $f$ be a stable degree $2$ polynomial with coefficients $\a_\gamma$, as above. Consider the set $\a = \{1,2,\dots,6r + 2\} \setminus (\{3,9,\dots, 6r_1 - 3\} \cup \{ 6, 12, \dots, 6r_2 \})$ for some $r \geq r_1,r_2\geq 0$. Then, a short calculation shows $f(\a) = (r_1+1) a_1 + (r_2+1) a_2$. Clearly, for any $\e>0$, there are $r_1, r_2$ such that $\norm{f(\a)} \leq \e$. Unfortunately, this succinct argument does not generalise well to $d \geq 3$.
	
\end{remark}

\begin{step}[Reduction to model setting]\label{step:2@thm:case-d-2}
	Let $\bar{f} \in \poly(\cFe \to G_\bullet/\Gamma)$ be a polynomial $\cS_2$-sequence with $\bar f(\emptyset) = e \Gamma$, where $G_\bullet/\Gamma$ is a length $2$ nilmanifold. Define the nilmanifold $\tilde G_\bullet/\tilde \Gamma$ by $\tilde G_0 = \tilde G_1 = \tilde G_2 = G_2$, $\tilde G_3 = \tilde G_4 = \dots = \{e_{G}\}$, $\tilde \Gamma = \Gamma \cap G_2$. 
	
	Then, $\tilde G_\bullet /\tilde \Gamma$ is a length $2$ nilmanifold, and there exists a polynomial sequence $\bar{g} \in \poly(\cFe \to \tilde G_\bullet/\tilde \Gamma)$ such that $\iota \circ \bar{g}$ is a $\cS_2$-asymptotic subsequence of $\bar f$, where $\iota \colon \tilde G/ \tilde \Gamma \to G/\Gamma$ is the natural inclusion $g \tilde \Gamma \mapsto g \Gamma$.
\end{step}
\begin{proof}
	It is clear that $\tilde G_\bullet$ is a filtration, and that $\tilde \Gamma$ is discrete. Each quotient $\tilde G / \tilde \Gamma \cap \tilde G_j$ 
	is either $G_2/\Gamma \cap G_2$ or trivial, 
	 so $\tilde G_\bullet$ is evidently $\tilde \Gamma$-rational. The content of this step is that $\bar{f}$ has an asymptotic subsequence which takes values in $G_2 \Gamma$. (See Lemma \ref{lem:poly-on-subnilmanifold} for details.)
	
	 Let $\tilde{\bar{f}}$ be a stable asymptotic $\cS_2$-subsequence $\bar{f}$. We claim that indeed $\tilde{\bar{f}}(\a) \in G_2 \Gamma$ for all $\a \in \cFe$. By definition, $\tilde{\bar{f}}$ is the projection of some $\tilde{f} \in \poly(\cFe \to G_\bullet)$. For any disjoint $\a,\b \in \cFe$, we have $\tilde f(\a \cup \b) = \tilde f(\a) \tilde f(\b) \Delta_\b \tilde f( \a)$ with $\Delta_\b \tilde f (\a) \in G_2$. 	 
	 
	 Because $G_2$ is normal, there is a well-defined projection map $\pi \colon G/\Gamma \to G/G_2\Gamma$ given by $g \Gamma \mapsto g \Gamma G_2$. The map $\bar h = \pi \circ \tilde{\bar{f}}$ is easily verified to be a degree $1$ polynomial: 
	 $$
\bar h( \a \cup \b) 
= \tilde f(\a \cup \b) G_2 \Gamma 
= \tilde f(\a) \tilde f(\b) \Delta_\b \tilde f( \a) G_2 \Gamma 
= \tilde f(\a) \tilde f(\b) G_2 \Gamma =  \bar h(\a)  + \bar h(\b). 
$$
We also have $\bar h(\emptyset) = eG_2 \Gamma$. Moreover, $\bar h$ is a stable $\cS_2$-sequence, because $\tilde{\bar{f}}$ is stable.   Repeating the argument from the case $d=1$, we find that $\bar h(\a) =  e G_2 \Gamma $ for all $\a \in \cFe$. Thus, $\tilde{\bar{f}}$ takes values in $G_2 \Gamma $, as needed.
\end{proof}

Combining the two steps easily finishes the proof. Start with a $\cS_2$-sequence  $\bar{f} \in \poly(\cFe \to G_\bullet/\Gamma)$ with $\bar f(\emptyset) = e \Gamma$, where $G_\bullet/\Gamma$ is a length $2$ nilmanifold. Let $\bar{g}$ be the asymptotic $\cS_2$-subsequence constructed Step \ref{step:2@thm:case-d-2}. Now, $\bar{g} \in \poly(\cFe \to \tilde G_\bullet/\tilde \Gamma)$, where $\tilde G_0 = G_2$ is abelian. We may without loss of generality suppose that $\tilde{G_\bullet}/\tilde \Gamma$ is a torus, equipped with the standard filtration of length $2$.

Applying Step \ref{step:1@thm:case-d-2} to $\bar{g}$, we find that $e \tilde \Gamma \in \cl{\set{ \bar{g}(\a)}{ \a \in \cS_2  }}$, whence $e \Gamma \in \cl{\set{ \bar{f}(\a) }{ \a \in \cS_2  }}$.
\end{proof}

\section{Main results}\label{sec:MAIN}

\subsection{Robust version and induction}\label{sec:MAIN:robust}
We now approach the proof of Theorem \ref{thm:B} for general $d \geq 1$. To begin with, we state a more robust version, which is better suited for an inductive proof.

\begin{theorem}[\ref{thm:B}, robust version]\label{thm:main-robust}
	Let $G_{\bullet}/\Gamma$ be a nilmanifold of length $d$, and let $\bar{f} \in \poly(\cFe \to G_{\bullet}/\Gamma)$ with $\bar f(\emptyset) = e \Gamma$.  Let $0 \leq k \leq 4d$ be an integer, and put $l = k - 4d$. Then, there exists a asymptotic $\cS_l$-subsequence of $\bar{f}$ which is constantly equal to $e \Gamma$.
\end{theorem}

Note that the strong version of Theorem \ref{thm:B}, \ref{thm:B-poly(F->G/Gamma)}, follows directly from the above statement by putting $k = 4d$. 

We will prove the above robust version of our main theorem by induction on the complexity of $G_{\bullet}/\Gamma$. The following lemma gives a useful description of polynomial sequences taking values on subnilmanifolds. We will often use this identification implicitly.

\begin{lemma}\label{lem:poly-on-subnilmanifold}
	Let $G_\bullet/\Gamma$ be a prenilmanifold of length $d$, and let $r \leq d$. Let the filtration $\tilde G_\bullet$ on $\tilde{G} = G_r$ be given by $\tilde G_i = G_r $ for $i \leq r$ and $\tilde G_i = G_i$ for $i \geq r$, and let $\tilde \Gamma := G_r \cap \Gamma$. Then, there is a natural bijective correspondence between maps $\bar f \in \poly(\cFe \to G_\bullet/\Gamma)$ taking values in $G_r \Gamma = \set{g \Gamma }{ g \in G_r} \subset G/\Gamma$ on one side and polynomial maps $\bar h \in \poly(\cFe \to \tilde G_\bullet/\tilde \Gamma)$ on the other side. More precisely, $\bar{h}$ corresponds to $\bar f = \iota \circ \bar{h}$, where $\iota\colon G_r/(G_r \cap \Gamma) \to G/\Gamma$ is the natural inclusion given by $g (G_r \cap \Gamma) \mapsto g \Gamma$. 
\end{lemma}
\begin{proof}

	Given $\bar h \in \poly(\cFe \to \tilde G_\bullet/\tilde \Gamma)$, we may lift it to $h \in \poly(\cFe \to \tilde G_\bullet)$ so that $\bar h (\a) = h(\a) \tilde \Gamma$. Since $\tilde G_0 = G_r \subset G_0$, there is a natural way to consider $h$ as a map to $G_0$, and under this identification $h \in \poly(\cFe \to G_\bullet)$. Thus, we may take $\bar{f} \in \poly(\cFe \to G_\bullet/\Gamma)$ given by $\bar{f}(\a) := h(\a) \Gamma$. It is clear that $\bar{f}$ takes values in $G_r\Gamma$, and $\bar f = \iota \circ \bar h$.
	
	For the other direction, assume that $\bar f \in \poly(\cFe \to G_\bullet/\Gamma)$ taking values in $G_r \Gamma$ is given. For each $\a \in \cFe$, there is unique $\bar{h}(\a) \in G_r/(G_r \cap \Gamma)$ such that $\iota(\bar{h}(\a)) = \bar{f}(\a)$: one can take arbitrary $g_r \in G_r$ such that $\bar{f}(\a) = g_r \Gamma$ and put $\bar{h}(\a) = g_r (G_r \cap \Gamma)$ (it is easy to see that $g_r$ is unique up to multiplication by and element of $G_r\cap \Gamma$). Hence, we have a map $\bar h \colon \cFe \to G_r/(G_r \cap \Gamma)$, and it remains to check that $\bar h$ is a polynomial. 
	
	Take any parallelepiped $(\a_\omega)_{\omega \in \Q{k}} \in \cFe^{[k]}$. Let us consider the cubes $\bar{\bb h} = (\bar h(\a_\omega))_\omega$ and $\bar{\bb g} = (\bar f(\a_\omega))_\omega \in \HK^{k}(G_\bullet/\Gamma)$. We claim that $\bar{ \bb g} $ can be lifted to a cube $\bb g = (g_\omega)_\omega \in \HK^{k}(G_\bullet)$ so that $g_\omega \in G_r$ for all $\omega$. Once this is accomplished, we can construe $\bb g$ as an element of $\HK^{k}(\tilde G_\bullet)$, so that $\bar{ \bb h}$ is the projection of $\bb g$, and in particular $\bar{ \bb h} \in \HK^{k}(\tilde G_\bullet/\tilde \Gamma)$.
	
	Begin with any lift $\bb g \in \HK^{k}(G_\bullet)$ of $\bar{\bb g}$, and recall that $\bb g$ can be written as
	$$
		\bb g = \prod_{\omega \in \Q{k}}^{\prec} \tilde g_{\omega}^{[\omega]} 
		= \tilde g_{\omega_1}^{[\omega_1]}  \tilde g_{\omega_2}^{[\omega_2]} \dots \tilde g_{\omega_{2^k}}^{[\omega_{2^k}]},
	$$
	where the product is taken in some fixed order compatible with the order induced by inclusion. Note that we are still free replace $\bb g$ by any element of $\HK^{k}(G_\bullet) \cap \Gamma^{\Q{k}}$ (by a slight abuse of notation, we retain the same symbol $\bb g$). We will inductively alter $\bb g$ in this way, so that $\tilde g_\omega \in G_r$, where we consider $\omega$ in the order of increasing size $\abs{\omega}$. Supposing that $g_\sigma \in G_r$ for $\sigma \prec \omega$, we may expand
	$$
		g_\omega = \prod_{\sigma \preceq \omega}^{\prec} \tilde{g}_{\sigma} 
		= \bra{ \prod_{\sigma \prec \omega}^{\prec} \tilde{g}_{\sigma} } \tilde g_\omega 
		\in G_r \tilde g_\sigma.
	$$
	If $\abs{\omega} \geq r$, then we automatically have $\tilde{g}_\omega \in G_{\abs{\omega}} \subset G_r$, whence also $g_\omega \in G_r$, so we are done. Otherwise, $\abs{\omega} < r$, and $g_\omega \in G_r \Gamma \cap G_{\abs{\omega}}$. Writing $g_{\omega} = g_{r}\gamma$ with $g_{r} \in G_r$, $\gamma \in \Gamma$, we have $\gamma \in \Gamma \cap G_{\abs{\omega}}$. After multiplying $\bb g$ by the inverse of $\gamma^{[\omega]} \in \HK^{k}(G_\bullet)$, we may assume that $\gamma = e$. (Note that this does not alter $\tilde{g}_\sigma$ unless $\sigma \succeq \omega$). Hence, we obtain $g_{\omega} \in G_r$ as desired.
\end{proof}

We will deduce Theorem \ref{thm:main-robust} from the following inductive step.

\begin{proposition}\label{prop:main-inductive-step}
	Let $\bar{f} \in \poly(\cF \to G_\bullet/\Gamma)$, where $G_\bullet/\Gamma$ is a nilmanifold of length $d$, and let $k \geq 0$ be an integer. Let $r$ be the first index such that $G_{r+1} \subsetneq G_0$. Put $l = k-2r$, and suppose that $l \geq 0$.
	
	Then, there exists an asymptotic $\cS_l$-subsequence $\bar{g}$ of $\bar{f}$ such that $g$ takes values in $G_{2r} \Gamma \simeq G_{2r}/(G_{2r} \cap \Gamma)$.
\end{proposition}

\begin{proof}[Proof of Theorem \ref{thm:main-robust}, assuming Proposition \ref{prop:main-inductive-step}]
	Take $\bar{f} \in \poly(\cF \to G_\bullet/\Gamma)$ as in Theorem \ref{thm:main-robust}. By repeated application of Proposition \ref{prop:main-inductive-step}, for $n \geq 0$ we construct asymptotic $\cS_{l_n}$-subsequences $\bar f^{(n)}$, where (up to the identification discussed in Lemma \ref{lem:poly-on-subnilmanifold}) $\bar f^{(n)} \in \poly(\cF \to G_\bullet^{(n)}/\Gamma^{(n)})$. We begin with $G^{(0)}_i = G_i$ and $\Gamma^{(0)} = \Gamma$, $l^{(0)} = k$ and $r^{(0)} \geq 1$. 
	
	For each $n \geq 0$, the nilmanifold $G_\bullet^{(n+1)}/\Gamma^{(n+1)}$ takes the form $G_0^{(n+1)} = G_1^{(n+1)} = \dots = G_{2r_n}^{(n+1)} = G_{2r_n}$ and $G_i^{(n)} = G_i$ for $i > 2r^{(n)}$, and $\Gamma^{(n+1)} = \Gamma \cap G^{(n+1)}_0$. We may finally take $l^{(n+1)} = l^{(n)} - 2 r^{(n)} = k - \sum_{j=0}^{n} 2r^{(j)} $. The construction ensures that $r^{(n+1)} \geq 2r^{(n)}$. The subsequence at step $n$ can be constructed, provided that $l^{(n)} \geq 0$.
		
	We are specifically interested in the first time $n$ such that $r^{(n)} \geq d+1$, whence $G^{(n)}$ is trivial and $f^{(n)}$ is constantly equal $e_G$. For such $n$, we have $2 r^{(n-1)} \leq 2d$, so $l^{(n)} \geq k - 4d \geq 0$ by assumption.
\end{proof}

\subsection{Reduction to an abelian problem}\label{sec:MAIN:abelian}
We deduce Proposition \ref{prop:main-inductive-step} from a model problem in the abelian setting. The proof of the following proposition is the most technical element of this paper, and occupies most of Section \ref{sec:MOD}.

\begin{proposition}\label{prop:main-abelian}
Let $k, d$ be integers with $d \leq k +1$, and put $l = k-d-1 \geq 0$. Let $f \colon \cFe \to \TT^m$ be a polynomial map of degree $d$ into a torus. 

Then, the constant sequence $0 \in \TT^m$ is an asymptotic $\cS_l$-subsequence of $f$.
\end{proposition}

\begin{remark}
	Any finite abelian group $A$ can be embedded in a torus, so the above proposition applies equally to polynomial maps $f \colon \cFe \to \TT^m \times A$. 
\end{remark}

\begin{proof}[Proof of Proposition \ref{prop:main-inductive-step} assuming Proposition \ref{prop:main-abelian}]
	Let $\bar{f} \in \poly(\cF \to G_\bullet/\Gamma)$ as in Proposition \ref{prop:main-inductive-step}. It  will suffice to find a $\cS_l$-asymptotic subsequence $\bar{g}$ of $\bar{f}$ which takes values in $G_{2r}\Gamma$.
	
	We have a natural projection map $\pi \colon G/\Gamma \to G/G_{2r} \Gamma  $, given by $g\Gamma \mapsto gG_{2r}\Gamma $. Since $[G_0,G_0] = [G_{r},G_{r}] \subset G_{2r}$, the quotient $G/G_{2r} \Gamma = (G/G_{2r})/(G_{2r}\Gamma /G_{2r})$ is a compact connected abelian Lie group, hence a torus.
	
	Note that $\pi \circ \bar{f} \colon \cFe \to G/\Gamma G_{2r}$ is a polynomial with respect to the length $2r$ filtration $\tilde{G}_j = G_j/G_{2r}$ for $j \leq 2r$ (and $\tilde G_j = \{e\}$ for $j > 2r$). We may identify $G/\Gamma G_{2r}$ as embedded in a torus, equipped with the standard length $2r-1$ filtration, and thus construe $\pi \circ \bar{f}$ as a degree $2r-1$ polynomial into a torus.
	
	 Hence, we may apply Proposition \ref{prop:main-abelian} to extract an asymptotic $\cS_{k-2r}$-subsequence of $\pi \circ \bar{f}$ which is identically $0$. By a standard diagonal argument, we may assume that this asymptotic $\cS_{k-2r}$-subsequence takes the form $\pi \circ \bar{g}$ where $\bar{g}$ is an asymptotic $\cS_{k-2r}$-subsequence of $\bar{f}$. This means precisely that $\bar{g}$ takes values in $G_{2r} \Gamma $, so we are done.
\end{proof}

\subsection{A counterexample}\label{sec:MAIN:c-exple}

It is not clear if Theorem \ref{thm:main-robust} is sharp, in the sense that it is no longer true for a larger value of $l$. Indeed, if the answer to Question \ref{question:basic} is positive, one would expect the theorem to hold with $l = k-d$. 

Here, we give an example showing that the bounds in Proposition \ref{prop:main-abelian} and in Theorem \ref{thm:A} are close to being sharp.

\begin{example}\label{MAIN:ex:d=k+1}
	Let $k$ be an integer, and put $d=k+1$. Consider the sequence
$$
	f(\a) = t \sum_{\substack{ \emptyset \neq \gamma \subset \a \\ \diam (\gamma) \leq k }} (-1)^{\abs{\gamma}} \pmod{1},
$$
where $t \in \TT \setminus \{0\}$ is a constant. Since any set $\gamma$ with $\diam (\gamma) \leq k$ automatically has $\abs{\gamma} \leq k+1$, thus defined $f$ is evidently a polynomial of degree $d$ with $f(\emptyset) = 0$.

We claim that for any $\a \in \cS_k$, we have $f(\a) = - t$. In particular, the constant $0$ sequence is not an asymptotic $\cS_0$-subsequence of $f$.

We proceed by induction on $\abs{\a}$. If $\abs{\a} = 1$ then we have $f(\a) = -t$ directly by definition. If $\abs{\a} \geq 2$, writing $i = \max \a$ and $\a = \a' \cup \{j\}$, we may compute that 
$$
	f(\a) = f(\a') - t \sum_{\substack{ \delta \subset \a \cap [j-k,j) }} (-1)^{\abs{\delta}} = f(\a') = -t,
$$
	by the inclusion-exclusion and the inductive assumption.
\end{example}

\begin{example} \label{MAIN:ex:l=k-d+2}
	Fix integers $d \leq k$. We construct a degree $d$ polynomial $f \in \poly(\cFe \to \TT)$ such that if $f$ is viewed as a $\cS_k$-sequence, then for $l \geq k-d+2$ the constant sequence $0$ is \emph{not} an asymptotic $\cS_l$-subsequence $f$. 
	
	 Generalising the construction from Example \ref{MAIN:ex:l=k-d+2}, let $f$ be given by
$$
	f(\a) = t \sum_{\substack{ \emptyset \neq \gamma \subset \a \\ \diam (\gamma) \leq k \\ \abs{\gamma} \leq d}} (-1)^{\abs{\gamma}} \pmod{1},
$$
where $t \in \TT \setminus \{0\}$. Ostensibly, $f$ is a polynomial of degree $d$ and $f(\emptyset) = 0$. Suppose now that $\FU(\a_i)$ is an \ip\ ring such $\b \mapsto \a_\b$ maps $\cS_{l}$ sets to $\cS_{k}$ sets, where $l$ is such as above. We will show that $f(\a_\b)$ is far from $0$ for some $\b \in \cS_l$. Indeed, we will show that $f(\a_i) = -t$ for sufficiently large $i$, namely $i$ sufficiently large that $\min \a_i > \max \a_1, \max \a_2, \dots, \max \a_l$.

Consider any $\gamma \subset \a_i$ with $\diam(\gamma) \leq k$. We will show these properties already imply that $\abs{\gamma} \leq d$. Let $n = \min \gamma$ so that $\gamma \subset [n,n+k]$. For any residue class $m \in [l]$, $m \not \equiv i \pmod{l}$, there needs to be some $j \equiv m \pmod{l}$ such that $\a_j \cap [n,n+k] \neq \emptyset$; else for $\b = \{m,m+l,m+{2l}, \dots, m+bl\}$ with large $b$, the set $\a_\b$ would fail to have gaps $\leq k$ although $\b$ has gaps $\leq l$. Thus, by a counting argument we have $\abs{\gamma} \leq (k+1) - (l-1) \leq d$. If follows that
$$
	f(\a_i) = -t + t \sum_{\substack{ \gamma \subset \a_i \\ \diam (\gamma) \leq k}} (-1)^{\abs{\gamma}} = -t \pmod{1},  
$$
where the last equality is a general fact about $\cS_k$ sets which is proved by simple induction on $\abs{\a_i}$.
\end{example}

\begin{example} \label{MAIN:ex:degree-2}
	Fix $k \geq 3$, and consider a ``generic'' degree $2$ sequence 
	$$
		f(\a) = \sum_{\substack{ \emptyset \neq \gamma \subset \a \\ \abs{\gamma} \leq 2 }} a_\gamma,
	$$
	where $a_\gamma \in \TT^{\binom{k+1}{2}}$ obey the following ``stability'' conditions: $a_{\gamma + k} = a_\gamma$, $a_\gamma = 0$ if $\operatorname{gap}(\gamma) > k$, $a_{i} = - a_{i,i+k }$, $a_{i,j} + a_{j,i+k} = 0$ and $\set{a_{i,j} }{ 1 \leq i \leq j \leq k}$ are linearly independent. (To shorten the notation, we write $a_{i,j}$ for $a_{\{i,j\}}$.)
	
	Put $l = k-1$. We will show that the constant $0$ sequence is not an asymptotic $\cS_l$-subsequence of $f$. More precisely, we claim that if $\tilde f( \b) = f(\a_\b)$ is an $\cS_k$-subsequence of $f$ and we write out the expansion 
	$$
		\tilde f(\b) = \sum_{\substack{ \emptyset \neq \delta \subset \b \\ \abs{\delta} \leq 2 }} b_\delta,
	$$
	then for infinitely many $j$, the coefficient $b_{j,j+l}$ is of a rather special form $a_{i_1,i_2} + \epsilon_1 a_{i_1',i_2} + \epsilon_2 a_{i_1,i_2'} \neq 0$, where $\epsilon_1,\epsilon_2 \in \{0,1\}$ and $i_1' \not \equiv i_1 \pmod{k}$, $i_2' \not \equiv i_2 \pmod{k}$.  Note that we have $b_{j,j+l} = \sum_{i \in \a_j, i' \in \a_{j+l}} a_{i,i'}.$
	
	Fix a sufficiently large $j$. We may assume that $\min \a_j < \min \a_{j+l}$ and $\max \a_j < \max \a_{j+l}$. Since the set $\a_j \cup \a_{j+l}$ has gaps bounded by $k$, there is the least $i_1 \in \a_j$ such that $[i_1,i_1 +k] \cap \a_{j+l} \neq \emptyset$. By the same token, there is the largerst $i_2 \in \a_{j+l}$ such that $[i_2-k,i_2]\cap \a_{j} \neq \emptyset$. 

	Each length $k$ interval $[a,a+k) \subset [i_1,i_2]$ contains an element of $\a_j,\a_{j+l}$, as well as an element of one of $\a_m$ for $m$ in each of the residue classes modulo $l$ different than $j$. By an inductive argument reminiscent of Observation \ref{SEQ:obs:form-of-Sk-sseq}, we see that $[i_1,i_2] \cap \a_j$ and $[i_1,i_2] \cap \a_j'$ are arithmetic progressions of step $k$ (same applies to $[i_1,i_2] \cap \bigcup_{n \equiv m \bmod{k}} \a_n$ for each $m \not \equiv j \pmod{k}$). Because of the stability conditions we imposed, the total contribution to $b_{j,j+l}$ from the part of $\a_j,\a_{j+l}$ contained in $[i_1,i_2]$ is
	$$
		\sum_{ \substack{ i \in [i_1,i_2] \cap \a_j,\\  i' \in  [i_1,i_2] \cap \a_{j+l}}} a_{i,i'} =
		 \sum_{\substack{  i < i' < i + k \\ i_1 \leq i, \ i' < i_2 \\ } }(a_{i,i'} + a_{i',i+k}) + a_{i_1^*,i_2}
		 = a_{i_1^*,i_2} = a_{i_1,i_2^*},
	$$	
	where $i_1^* = \max [i_1,i_2] \cap \a_{j}$ and $i_2^* = \min [i_1,i_2] \cap \a_{j+l}$.
	
	By another counting argument, we see that the interval $[i_2^*-k,i_{2}^*)$ may contain at most one additional element $i_1' < i_1$ of $\a_{j}$, and likewise interval $(i_1^*,i_{1}^*+k]$ may contain one additional element $i_2'$ of $\a_{j+l}$. These give the contributions $a_{i_1',i_2^*}$ and $a_{i_1^*,i_2'}$ accordingly. We may finally bring the result to the required form, using periodicity of $a_\gamma$.
\end{example}

\begin{remark}
	We expect that for any $d \leq k+1$, there exists a degree $d$ polynomial $\cS_k$-sequence which does not have the constant $0$ sequence as asymptotic $\cS_l$-subsequence for $l \geq k-d+1$. However, exhibiting concrete examples of such sequences proves problematic. 
\end{remark}

\section{Model problem}\label{sec:MOD}

In this section we prove Proposition \ref{prop:main-abelian}. Together with previous considerations, this will finish the proof of our main result, Theorem \ref{thm:B}. Towards the end, we also explain how to adapt the argument to prove Theorem \ref{thm:A}.

Let $f \colon \cFe \to \TT$ be a degree $d$ polynomial viewed as a $\cS_k$-subsequence, as in  Proposition \ref{prop:main-abelian}. Suppose further that $k \geq d +1$ and let $l = k-d-1$. Passing to an asymptotic $\cS_k$-subsequence if necessary, we may assume that $f$ is stable. Recall that by Corollary \ref{cor:structure-of-stable-poly(F->T)}, $f$ takes the form
\begin{align}
	f(\b) &= \sum_{\gamma \subset \beta} a_\gamma,  \label{MOD:eq:001} \\
	a_{\gamma} &= 0 \text{ if } \abs{\gamma} \geq d \text{ or } \diam \gamma > k.
\end{align}

Our main idea, much as in the case $d=2$ discussed in Section \ref{sec:SPL}, is to begin with a suitably generic \ip\ ring $\FU(\a_i)$ such that $\b \mapsto a_\beta$ maps $\cS_l$ to $\cS_k$, 
and to perturb $\a_i$ slightly to ensure that $f(\a_\b)$ is small for all $\b \in \cFe$. Throughout, $(\a_i)$ denotes a sequence of disjoint $\cS_k$ sets.

\subsection{Patterns}\label{sec:MOD:pattern}
We define a \emph{pattern of length $M$} to be a collection of disjoint sets $ \pi = (\pi_i)_{i=1}^M$ which are either in $\cS_k$ or empty, such that $\beta \mapsto \pi_\beta$ maps $\cS_l$ sets to $\cS_k$, provided that $\pi_i$ are all non-empty for $i \in \beta$. We say that $(\pi_i)$ \emph{occurs} at position $n \equiv 0 \pmod{k}$ in $(\a_i)$ if for each $i$, we have $(\a_i - n) \cap [M] = \pi_i$. (Note that this definition depends on $M$ as well as on $\a_i$ and $\pi_i$).

We say that $(\a_i)$ is \emph{well-formed} if the following conditions are satisfied:
\begin{enumerate}
\item for each $i$, $\min \a_i \equiv \max \a_i \pmod{k}$,
\item for each $i$, $\min \a_{i+l+1} > \max \a_i + k$,
\item $\b \mapsto \a_\b$ maps $\cS_l$ to $\cS_k$.
\end{enumerate}
(Note that we are working with fixed $k,l$ and this definition is specific to those values.)

We say that $(\a_i)$ is $(M,N)$-generic if for any pattern $\pi$ of length at most $M$, one of the following holds:
\begin{enumerate}
\item there is a constant $C = C(\pi)$ such that $(\pi_i)$ occurs at most $C$ times in any well-formed sequence $(\b_i)$,
\item the pattern $\pi$ occurs at least $N$ times in $(\a_i$).
\end{enumerate}
We apply this definition in the regime where $M$ is fixed and $N \to \infty$. In fact, it  is enough to take $M = 3k$. The following observation shows that the above definition is not vacuous. 

\begin{claim}\label{claim:generic-seqs-exist}
For any $(M,N)$, there exists a well-formed sequence of sets $(\a_i)$ which is $(M,N)$-generic.
\end{claim}
\begin{proof}
	We need to construct $(\a_i)$ such that each pattern $(\pi_i)$ of bounded length, which may potentially occur numerous times in some well-formed $(\b_i)$, occurs many times in $(\a_i)$. Our strategy is to begin with a class of patterns whose numerous occurrences may be easily guaranteed, and then gradually extending this class. The construction is (implicitly) inductive, but we reuse the same symbol $(\a_i)$ at each step.
	
	Taking $\a_i$ to be long arithmetic progressions with step $k$, is easy to ensure that $(\a_i)$ has many occurrences of any pattern of the form
	\begin{equation}
	\pi_i =
		\begin{cases}
		 \{ (i \bmod k) , (i \bmod k)  + k, \dots, (i \bmod k)  + mk \}, &  \text{if } j \leq i \leq i+l' \\
		\emptyset, & \text{otherwise}, 
		\end{cases}
		\label{MOD:eq:002}
	\end{equation}
	where $l' = l$ or $l'=l-1$, and $m$ is bounded. It is not difficult to alter $(\a_i)$ so as to change any occurrence of the pattern \eqref{MOD:eq:002} into a pattern of the similar form
	\begin{equation}
	\pi_i =
		\begin{cases}
		 \{ a_i , a_{i} + k, a_{i} + 2k, \dots, a_{i} + m'k \}, &  \text{if } j \leq i \leq i+l' \\
		\emptyset, & \text{otherwise}, 
	\end{cases}
		\label{MOD:eq:003}
	\end{equation}	
	where $1 \leq a_i \leq k$ are arbitrary and $m' = m - O(1)$. Hence, we may assume that $(\a_i)$ contains many occurrences of patterns \eqref{MOD:eq:003}. Applying similar reasoning, we may also convert occurrences of \eqref{MOD:eq:003} into patterns 
	\begin{equation}
	\pi_i =
		\begin{cases}
		 \{ a_i , a_{i} + k, \dots , a_i + t_i k, b_i + t_i k, \dots, b_{i} + m''k \}, &  \text{if } j \leq i \leq i+l' \\
		\emptyset, & \text{otherwise}, 
	\end{cases}
		\label{MOD:eq:004}
	\end{equation}	
	where $1 \leq b_i \leq k$ are arbitrary and $m'' = m'-O(1)$ (but we claim no control over the $t_i$). 
	
	Finally, take any pattern $(\pi_i)$ which occurs numerous times in some well-formed $(\b_i)$. There is some index $j$ such that if $\pi_i \neq \emptyset$ then $j \leq i \leq j+l$. We may further assume that $\pi_i \neq \emptyset$ precisely 	for $j \leq i \leq j+l'$, where $l' = l$ or $l' = l-1$. Letting $a_i = \min \pi_i \bmod k$ and $b_i = \max \pi_i \bmod k$, we see that any occurrence of the pattern \eqref{MOD:eq:004} may be converted into an occurrence of the sought pattern $\pi$, supposing (as we may) that $m''$ is sufficiently large with respect to $\pi$.
\end{proof}

\subsection{Perturbations}\label{sec:MOD:perturbation}
Fix an arbitrary metric on $\TT^{\cFe}$ which is compatible with the product structure. We define $\Sigma$ to be the set of those $s \in \TT^{\cFe}$ such that for any $\e>0$ and any $N$ there exists $(3k, N)$-generic, well-formed $(\a_i)$ such that $\norm{ f(\a_\xi) - s_{\xi} } < \e$ (where $\xi$ stands for a dummy variable\footnote{Hence, strictly speaking we mean $\norm{ \bra{ f(\a_\xi) - s_{\xi} }_{\xi \in \cFe} } < \e$.}).

We further define the set of possible ``perturbations'' $\Delta \subset \TT^{\cFe}$.
For $t = (t_\xi)_{\xi \in \cFe}$, we declare $t \in \Delta_0$ if and only if there exists constants $L$ and $N_0$ such that for all $N \geq N_0$, all $(3k,N)$-generic, well-formed $(\a_i)$, we can find well-formed $(\a'_i)$ with $ \sum_{i = 1}^\infty \abs{\a_i \triangle \a'_i} \leq L$ and $f(\a'_\xi) = f(\a_\xi) + t_\xi$. Finally, put $\Delta := \cl(\Delta_0)$.

Because of \eqref{MOD:eq:001}, any $(s_\xi) \in \Sigma$ admits a representation $s_\xi = \sum_{\gamma \subset \xi} c_\gamma$ for some constants $c_\gamma$. More precisely, if $s_\xi = f(\a_\xi)$, then $c_\gamma$ are given by $c_\gamma = \sum_{\delta}{a_{\delta}}$, where the sum runs over $\delta$ with $\delta \subset \a_\gamma $ but $\delta \not\subset \a_{\gamma'}$ for $\gamma' \subsetneq \gamma$. Stability of $f$ implies that $c_\gamma = 0$ if $\abs{\gamma} > d$ or  $\diam(\gamma) > k$. Likewise, any $t \in \Delta$ admits a representation of the same form $t_\xi = \sum_{\gamma \subset \xi} b_\gamma$, where $b_\gamma = 0$ if $\abs{\gamma} > d$ or  $\diam(\gamma) > k$. Let $A \subset \TT^{\cFe}$ be the set of all $(c_\gamma)_\gamma$ as above, and let $B$ accordingly be the set of all $(b_\gamma)_\gamma$.

The sets $\Sigma$ and $A$ are rather closely connected. Indeed, the map $\TT^{\cFe} \to \TT^{\cFe}$ given by $\{\xi \mapsto c_\xi \} \mapsto \{ \xi \mapsto \sum_{\gamma \subset \xi} c_\xi \}$ is a bijection (as verified by an inclusion-exclusion argument) and an isomorphism of groups. Same applies to $\Delta$ and $B$.

We make some simple observations concerning the sets just defined. Just as before, we have $\Delta + \Sigma = \Sigma$ and $\Delta + \Delta = \Delta$, and moreover $\Delta$ is a closed group. The argument is essentially the same as in Section \ref{sec:SPL:d=2}. For the same reasons, we have that $A + B = A$ and $B$ is a closed group. We now study $\Delta$ and $\Sigma$ in more detail. 

\begin{claim}\label{MOD:claim:1}
For any $\kappa,\gamma \in \cFe$ with $d \geq \abs{\gamma} \geq \abs{\kappa}$ and $\diam(\kappa) \leq l$, $\diam(\gamma) \leq k$, there exists $t = (t_\xi)_\xi \in \Delta$ such that
\begin{enumerate}
\item\label{item:1@claim:1@thm:model}
	$t_\beta = 0$ unless $\b \supseteq \kappa$,
\item\label{item:2@claim:1@thm:model}
	$t_\kappa \in a_\gamma + \sum_{\delta \supsetneq \gamma} \ZZ a_\delta$.
\end{enumerate}
\end{claim}
\begin{proof}

	Without loss of generality, we may assume that $\gamma \subset [k,2k]$. We begin by choosing a pattern $\pi = (\pi_j)$ of length $3k$ such that
\begin{enumerate}
\item $\pi_j \neq \emptyset$ for $j \in \kappa$, 
\item $\bigcup_{j =1}^\infty \pi_j \cap \gamma = \emptyset$,
\item $(\pi_j)$ occurs in all $(3k,N)$-generic $(\a_i)$ for sufficiently large $N$.
\end{enumerate}	

	Such $\pi$ can be constructed greedily, assigning each $m \in [3k] \setminus \gamma$	to consecutive sets $\pi_j$, $j = \min \kappa, \dots, \min \kappa + l$. (It is at this point that we are using that $l + d + 1 \leq k$.)

	Since $\abs{\gamma} \geq \abs{\kappa}$, we may partition $\gamma = \bigcup_{j \in \kappa} \gamma_{j}$ into $\abs{\kappa}$ non-empty sets. Let $\gamma_j = \emptyset$ for $j \not \in \kappa$. Fix one such partition once and for all.
	
	Suppose that $(\a_i)$ is a well-formed, $(3k,N)$-generic sequence of sets, for some large $N$. Pick, arbitrarily, some $n$ such that $(\pi_j)$ appears at position $n$. For any $\sigma \subset \gamma$, we may consider the distortion of $(\a_i)$ given by $\a^\sigma_i = \a_i \cup (\gamma_i \cap \sigma + n)$. Note that the union is disjoint, and we have:
	$$
		(\a_i^\sigma - n) \cap[ 3k ] = \pi_i \cup (\gamma_i \cap \sigma).
	$$
				
	 Clearly, thus obtained $\a_i^\sigma$ differs from $\a_i$ only in boundedly many places. Hence, the difference ${ f(\a_\xi^\sigma) - f(\a_\xi) }$ belongs to $\Delta$. Note that this difference depends only on $\pi$ and $\sigma$, but not on $\a$ (this is the reason why we work with patterns of length $3k$ rather that $k$). More generally, summing over all $\sigma$ (with appropriate choice of signs) we find:
$$
	t = (t_\xi)_{\xi} \in \Delta, \qquad t_\xi = \sum_{\sigma \subset \gamma} (-1)^{\abs{\sigma}} f(\a_{\xi}^\sigma).
$$	

	We now study the coefficients $t_\b$ for different sets $\b$. (Only $\beta \subseteq \kappa$ will play an important role.) We have
\begin{align*}
	t_\beta 
	&= \sum_{\sigma \subset \gamma} (-1)^{\abs{\sigma}} 
		f(\a_{\beta}^\sigma) 
	= \sum_{\sigma \subset \gamma} (-1)^{\abs{\sigma}}
		\sum_{\delta \subset \a_{\beta}^\sigma } a_\delta
	= \sum_{ \delta \subset \a^\gamma_\beta } 
	a_\delta 
		\sum_{
		\substack{\sigma \subset \gamma\\
			\delta \subset \a_{\beta}^\sigma }
		}
	(-1)^{\abs{\sigma}}.
\end{align*}
For fixed $\delta \subset \a^\kappa_\beta$, denote by $\sigma(\delta)$ the unique minimal set $\sigma \subset \gamma$ such that $\delta \subset \a^\sigma_\beta$ ($\sigma(\delta)$ does not depend on $\beta$). We now have
	\begin{align*}
	t_\beta &=
	\sum_{ \delta \subset \a^\gamma_\beta }  a_\delta 
		\sum_{\sigma(\delta) \subset \sigma \subset \gamma}
		(-1)^{\abs{\sigma}} .
	\end{align*}
	
	The inner sum is identically $0$, unless $\sigma(\delta) = \gamma$, when it is equal to $(-1)^{\abs \gamma}$. Note that we always have $\sigma(\delta) \subset \gamma_\beta$; indeed, $\a_\b^\sigma = \a_\b^{\sigma \cap \gamma_\beta}$. Hence, $t_\b = 0$ unless $\kappa \subset \beta$. When $\kappa \subset \beta$, we may rewrite the requirement $\sigma(\delta) = \gamma$ as $\gamma + n \subset \delta$ (here, $n$ is the index where $(\pi_i)$ appears in $(\a_i)$), and hence
	\begin{align*}
	t_\beta &=
	(-1)^{ \abs{\gamma} } \sum_{
		( \gamma + n) \subset  \delta \subset \a_\beta \cup ( \gamma + n) }
		a_\delta .
	\end{align*}
	We are specifically interested in $\b = \kappa$. It is clear that the above sum includes $a_\gamma = a_{\gamma + n}$, and all the remaining summands take the form $a_\delta$ with $\delta \supsetneq \gamma$. Thus, after multiplying by $(-1)^{\abs{\gamma}}$, the constructed $(t_\xi)$ satisfies the required conditions. 
\end{proof}

Let $\vec e_\kappa$ denote the ``base vector'' $(e_{\kappa,\xi})$ with $e_{\kappa,\kappa} = 1$ and $e_{\kappa,\xi} = 0$ for $\xi \neq \kappa$.

\begin{claim}\label{MOD:claim:3}
	The set $A$ (and hence also $B$) is contained in
	$$S =  \overline{\operatorname{span}}\set{ a_\gamma \vec e_\kappa }{ d \geq \abs{ \gamma } \geq \abs{\kappa}, \ \diam(\gamma) \leq k, \ \diam(\kappa) \leq l } ,$$
the smallest closed subgroup containing all elements  $a_\gamma \vec e_\kappa$.
\end{claim}
\begin{proof}
	Fix some well-formed $(\a_i)_i$, and write $f(\a_\xi)$ as $f(\a_\xi) = \sum_{\kappa \subset \xi}{c_\kappa}$. One can check that this formula holds with $c_\kappa$ given by 
	$$
		c_\kappa = \sum_{\substack{ \gamma \subset \a_\kappa \\ \gamma \not \subseteq \a_{\lambda} \text{ for } \lambda \subsetneq \kappa} } a_\gamma.
	$$ 
	The above sum only contains elements $a_\gamma$ with $\abs{\gamma} \geq \abs{\kappa}$, and we may eliminate the terms with $\abs{\gamma} > d$ because these vanish. Finally, if $\diam(\kappa) > l$ then (because $(\a_i)$ is well-formed), we have $\diam(\gamma) > k$ for all $\gamma$ in the above sum, and hence the corresponding $a_\gamma$ vanish. 
	
	It follows that $(c_\xi)_{\xi} \in S$. Since $S$ is closed, also $A \subset S$. Finally, $S$ is a group, so $A+B = A$, $A \subset S$ implies $B \subset S$.
\end{proof}

We next bootstrap Claim \ref{MOD:claim:1} to a more precise statement.

\begin{claim}\label{MOD:claim:2}
In the situation of Claim \ref{MOD:claim:1}, there exists $b = (b_\xi)_\xi \in B$ such that
\begin{enumerate}
\item\label{item:1@claim:2@thm:model}
	$b_\beta = 0$ unless $\b = \kappa$,
\item\label{item:2@claim:2@thm:model}
	$b_\kappa = a_\gamma$.
\end{enumerate}
\end{claim}
\begin{proof}
\renewcommand{\k}{\kappa}
\renewcommand{\l}{\lambda}

	It follows from Claim \ref{MOD:claim:1} that we may find for any $\gamma, \kappa$ with $\abs{\gamma} \geq \abs{\kappa}$ and $\diam(\kappa) \leq l$ a vector $(t_\xi)_\xi \in \Delta$ such that $$t_\kappa = a_\gamma + \sum_{\delta \supsetneq \gamma} k_\delta a_\delta$$ for some integers $k_\delta$. Applying  this inductively with $\delta$ in place of $\gamma$ and using that $a_\delta = 0$ for sufficiently large $\delta$, we may eliminate all the summands $a_\delta$ above, so that $t_\kappa = a_\gamma$. 	Writing out $t_\kappa$ in the coordinates, we thus find $b = (b_\xi)_\xi \in B$ such that $b_\b = 0$ unless $\b \supseteq \kappa$, and $b_\kappa = a_\gamma$. 
	
	Using a form of Gaussian elimination together with Claim \ref{MOD:claim:3}, we may now produce for any $\e > 0$ an element $b \in B$ such that $b_\kappa = a_\gamma$ and $\norm{ b_\b } < \e$ if $\b \not \supset \kappa$ or $\beta \in \{\k_1,\k_2,\dots,\k_r\}$, for any finite list $\{\k_j\}_{j=1}^r$. Finally, using the compactness of $B$, we may use the above procedure to produce  $b \in B$ such that $b_\kappa = a_\gamma$ and  $b_\b = 0$ if $\b \neq \k$. 
\end{proof}

\subsection{Final step}\label{sec:MOD:final} We are now in position to finish the proof of Proposition \ref{prop:main-abelian}. Combining Claims \ref{MOD:claim:2} and \ref{MOD:claim:3}, we see that $A = B$, and hence $\Sigma = \Delta$. But this means that $\Sigma$ contains the constant sequence $\b \mapsto 0$. Hence, there is some $(\a_i)$ such that $\b \mapsto \a_\b$ maps $\cS_l$ to $\cS_k$ (which additionally happens to be well-formed and generic) so that $\beta \mapsto f(\a_\beta)$ is as close to the constant sequence $\beta \mapsto 0$ (in the product topology of $\TT^{\cFe}$) as we wish. Thus, we can extract an asymptotic $\cS_l$-subsequence of $f$ which is identically $0$, which was our goal.

\subsection{Proof of Theorem \ref{thm:A}}\label{sec:MOD:Thm-A} Having proved Proposition \ref{prop:main-abelian} (and hence Theorem \ref{thm:B}) we discuss the proof of Theorem \ref{thm:A}. We are in the same situation as in Proposition \ref{prop:main-abelian}, with the exception that we have $k = d$, and we need to take $l = 0$.

We apply the same argument as above, with some simplifications. A pattern is now (again) just a single set $\pi \in \cS_k$. The appropriate version of Claim \ref{claim:generic-seqs-exist} is easily proved. The key difference in Claim \ref{MOD:claim:1} is that we can now prove it with $k = d$ and $l = 0$ (we may simply put $\pi = [3k] \setminus \gamma$, with notation therein). Neither Claim \ref{MOD:claim:3}, not the remainder of the argument ever use the relation between $k,l$ and $d$, so the reasoning carries through.

\bibliographystyle{alpha}
\bibliography{bibliography}

\begin{thebibliography}{GMV16b}

\bibitem[AGH63]{AuslanderGreenHahn1963}
L.~Auslander, L.~Green, and F.~Hahn.
\newblock {\em Flows on homogeneous spaces}.
\newblock With the assistance of L. Markus and W. Massey, and an appendix by L.
  Greenberg. Annals of Mathematics Studies, No. 53. Princeton University Press,
  Princeton, N.J., 1963.

\bibitem[Ber10]{Bergelson2010d}
Vitaly Bergelson.
\newblock Ultrafilters, {IP} sets, dynamics, and combinatorial number theory.
\newblock In {\em Ultrafilters across mathematics}, volume 530 of {\em Contemp.
  Math.}, pages 23--47. Amer. Math. Soc., Providence, RI, 2010.

\bibitem[BFM96]{BergelsonFurstenbergMcCutcheon1996}
Vitaly Bergelson, Hillel Furstenberg, and Randall McCutcheon.
\newblock I{P}-sets and polynomial recurrence.
\newblock {\em Ergodic Theory Dynam. Systems}, 16(5):963--974, 1996.

\bibitem[BFW06]{BergelsonFurstenbergWeiss2006}
Vitaly Bergelson, Hillel Furstenberg, and Benjamin Weiss.
\newblock Piecewise-{B}ohr sets of integers and combinatorial number theory.
\newblock In {\em Topics in discrete mathematics}, volume~26 of {\em Algorithms
  Combin.}, pages 13--37. Springer, Berlin, 2006.

\bibitem[BL03]{BergelsonLeibman-2003}
V.~Bergelson and A.~Leibman.
\newblock Topological multiple recurrence for polynomial configurations in
  nilpotent groups.
\newblock {\em Adv. Math.}, 175(2):271--296, 2003.

\bibitem[BL07]{BergelsonLeibman2007}
Vitaly Bergelson and Alexander Leibman.
\newblock Distribution of values of bounded generalized polynomials.
\newblock {\em Acta Math.}, 198(2):155--230, 2007.

\bibitem[BL16]{BergelsonLeibman2016}
Vitaly Bergelson and Alexander Leibman.
\newblock {$\mathrm{IP_r^*}$ recurrence and nilsystems, preprint}, 2016.

\bibitem[CS10]{SzegedyCamarena2010}
Omar~Antolin Camarena and Balazs Szegedy.
\newblock {Nilspaces, nilmanifolds and their morphisms}.
\newblock September 2010.

\bibitem[Fur81]{Furstenberg1981}
H.~Furstenberg.
\newblock {\em Recurrence in ergodic theory and combinatorial number theory}.
\newblock Princeton University Press, Princeton, N.J., 1981.
\newblock M. B. Porter Lectures.

\bibitem[FW78]{FurstenbergWeiss1978}
H.~Furstenberg and B.~Weiss.
\newblock Topological dynamics and combinatorial number theory.
\newblock {\em J. Analyse Math.}, 34:61--85 (1979), 1978.

\bibitem[GMV16a]{GutmanMannersVarju-1}
Yonatan Gutman, Freddie Manners, and P{\'e}ter~P. Varj{\'u}.
\newblock {The structure theory of Nilspaces, I}.
\newblock 2016.
\newblock (Preprint).

\bibitem[GMV16b]{GutmanMannersVarju-2}
Yonatan Gutman, Freddie Manners, and P{\'e}ter~P. Varj{\'u}.
\newblock {The structure theory of Nilspaces, II: Representaion as
  nilmanifolds}.
\newblock 2016.
\newblock (Preprint).

\bibitem[GMV16c]{GutmanMannersVarju-3}
Yonatan Gutman, Freddie Manners, and P{\'e}ter~P. Varj{\'u}.
\newblock {The structure theory of Nilspaces, III: Inverse limit
  representations and topological dynamics}.
\newblock 2016.
\newblock (Preprint).

\bibitem[Gre16]{Green-book}
Ben Green.
\newblock {\em {Higher-Order Fourier Analysis, I}}.
\newblock Oxford, 2016.
\newblock (Notes available from the author).

\bibitem[GT10]{GreenTao2010}
Ben Green and Terence Tao.
\newblock Linear equations in primes.
\newblock {\em Ann. of Math. (2)}, 171(3):1753--1850, 2010.

\bibitem[GT12]{GreenTao2012}
Ben Green and Terence Tao.
\newblock The quantitative behaviour of polynomial orbits on nilmanifolds.
\newblock {\em Ann. of Math. (2)}, 175(2):465--540, 2012.

\bibitem[GTZ12]{GreenTaoZiegler-2012}
Ben Green, Terence Tao, and Tamar Ziegler.
\newblock An inverse theorem for the {G}owers {$U\sp {s+1}[N]$}-norm.
\newblock {\em Ann. of Math. (2)}, 176(2):1231--1372, 2012.

\bibitem[Hin74]{Hindman1974}
Neil Hindman.
\newblock Finite sums from sequences within cells of a partition of {$N$}.
\newblock {\em J. Combinatorial Theory Ser. A}, 17:1--11, 1974.

\bibitem[HK05]{HostKra2005}
Bernard Host and Bryna Kra.
\newblock Nonconventional ergodic averages and nilmanifolds.
\newblock {\em Ann. of Math. (2)}, 161(1):397--488, 2005.

\bibitem[HK11]{HostKra2009}
Bernard Host and Bryna Kra.
\newblock Nil-{B}ohr sets of integers.
\newblock {\em Ergodic Theory Dynam. Systems}, 31(1):113--142, 2011.

\bibitem[HS12]{HindmanStrauss-book}
Neil Hindman and Dona Strauss.
\newblock {\em Algebra in the {S}tone-\v {C}ech compactification}.
\newblock de Gruyter Textbook. Walter de Gruyter \& Co., Berlin, second
  edition, 2012.

\bibitem[HSY16]{HuangShaoYe2014}
Wen Huang, Song Shao, and Xiangdong Ye.
\newblock Nil {B}ohr-sets and almost automorphy of higher order.
\newblock {\em Mem. Amer. Math. Soc.}, 241(1143):v+83, 2016.

\bibitem[Key66]{Keynes1966}
Harvey~B. Keynes.
\newblock Topological dynamics in coset transformation groups.
\newblock {\em Bull. Amer. Math. Soc.}, 72:1033--1035, 1966.

\bibitem[Key67]{Keynes1967}
Harvey~B. Keynes.
\newblock A study of the proximal relation in coset transformation groups.
\newblock {\em Trans. Amer. Math. Soc.}, 128:389--402, 1967.

\bibitem[Laz54]{Lazard-1954}
Michel Lazard.
\newblock Sur les groupes nilpotents et les anneaux de {L}ie.
\newblock {\em Ann. Sci. Ecole Norm. Sup. (3)}, 71:101--190, 1954.

\bibitem[Lei98]{Leibman1998}
A.~Leibman.
\newblock Polynomial sequences in groups.
\newblock {\em J. Algebra}, 201(1):189--206, 1998.

\bibitem[Lei02]{Leibman2002}
A.~Leibman.
\newblock Polynomial mappings of groups.
\newblock {\em Israel J. Math.}, 129:29--60, 2002.

\bibitem[Lei05a]{Leibman-2005b}
A.~Leibman.
\newblock Pointwise convergence of ergodic averages for polynomial actions of
  {${\Bbb Z}\sp d$} by translations on a nilmanifold.
\newblock {\em Ergodic Theory Dynam. Systems}, 25(1):215--225, 2005.

\bibitem[Lei05b]{Leibman2005}
A.~Leibman.
\newblock Pointwise convergence of ergodic averages for polynomial sequences of
  translations on a nilmanifold.
\newblock {\em Ergodic Theory Dynam. Systems}, 25(1):201--213, 2005.

\bibitem[Mal51]{Malcev1951}
A.~I. Malcev.
\newblock On a class of homogeneous spaces.
\newblock {\em Amer. Math. Soc. Translation}, 1951(39):33, 1951.

\bibitem[McC99]{McCutcheon1999}
Randall McCutcheon.
\newblock An infinitary polynomial van der {W}aerden theorem.
\newblock {\em J. Combin. Theory Ser. A}, 86(2):214--231, 1999.

\bibitem[Tao12]{Tao-book}
Terence Tao.
\newblock {\em Higher order {F}ourier analysis}, volume 142 of {\em Graduate
  Studies in Mathematics}.
\newblock American Mathematical Society, Providence, RI, 2012.

\bibitem[Tu14]{Siming}
Siming Tu.
\newblock {${\rm Nil \,Bohr}_0$}-sets and polynomial recurrence.
\newblock {\em J. Math. Anal. Appl.}, 409(2):890--898, 2014.

\bibitem[TV10]{TaoVu}
Terence Tao and Van~H. Vu.
\newblock {\em Additive combinatorics}, volume 105 of {\em Cambridge Studies in
  Advanced Mathematics}.
\newblock Cambridge University Press, Cambridge, 2010.
\newblock Paperback edition [of MR2289012].

\bibitem[ZK13]{Zorin-Kranich2013}
Pavel Zorin-Kranich.
\newblock {\em {Ergodic theorems for polynomials in nilpotent groups}}.
\newblock PhD thesis, Universiteit van Amsterdam, September 2013.

\bibitem[ZK14]{Zorin-Kranich2012}
Pavel Zorin-Kranich.
\newblock A nilpotent {IP} polynomial multiple recurrence theorem.
\newblock {\em J. Anal. Math.}, 123:183--225, 2014.

\end{thebibliography}

\end{document}